\documentclass[12pt]{article}
\usepackage{color}
\usepackage{amsmath}

\usepackage{extarrows}
\usepackage{geometry}
\usepackage{authblk} 
\usepackage{hyperref} 
\usepackage{color}
\usepackage{ntheorem}
\usepackage{enumitem}

\usepackage[mathscr]{euscript}
\usepackage[utf8]{inputenc}
\def\q{\hfill\rule{1ex}{1ex}}
\def\0{\emptyset}

\def\q{\hfill\rule{1ex}{1ex}}
\def\QEDopen{{\hfill\setlength{\fboxsep}{0pt}\setlength{\fboxrule}{0.2pt}\fbox{\rule[0pt]{0pt}{1.3ex}\rule[0pt]{1.3ex}{0pt}}}}

\newtheorem{theorem}{Theorem}[section]

\newtheorem{lemma}[theorem]{Lemma}

\newtheorem{cor}[theorem]{Corollary}

\newenvironment{provekplF}{{\noindent\it Proof of Theorem  \ref{thm: ex(n,kpl,F)}}.}{\hfill $\square$\par}
\newenvironment{provekplF=1}{{\noindent\it Proof of Theorem  \ref{thm: when beta F=1}}.}{\hfill $\square$\par}
\newenvironment{proof}{{\noindent\it Proof.}}{\hfill $\square$\par}

\usepackage{graphicx}

\usepackage{enumerate}
\usepackage{enumitem}

\usepackage{
	amsmath,			
	amssymb,			
	enumerate,		    
	graphicx,			
	lastpage,			
	multicol,			
	multirow,			
	pifont,			    
}

\usepackage[numbers]{natbib}

\newcommand{\lf}{\left\lfloor}
\newcommand{\rf}{\right\rfloor}
\newcommand{\lc}{\left\lceil}
\newcommand{\rc}{\right\rceil}
\newcommand{\lchuto}{\lf\frac{\ell}{2}\rf}
\newcommand{\lchutoo}{\lc\frac{\ell}{2}\rc}
\newcounter{cases}
\newcounter{subcases}[cases]

\begin{document}


\title{Tur\'an type problems for a fixed graph and a linear forest}
\author{
   {\small\bf Haixiang Zhang}\thanks{email:  zhang-hx22@mails.tsinghua.edu.cn}\quad
    {\small\bf Xiamiao Zhao}\thanks{Corresponding author: email:  zxm23@mails.tsinghua.edu.cn}\quad
    {\small\bf Mei Lu}\thanks{email: lumei@tsinghua.edu.cn}\\
    {\small Department of Mathematical Sciences, Tsinghua University, Beijing 100084, China.}\\
}

\date{}

\maketitle\baselineskip 16.3pt

\begin{abstract}
Let $\mathscr{F}$ be a family of graphs. A graph $G$ is $\mathscr{F}$-free if  $G$ does not contain any $F\in \mathscr{F}$  as a subgraph. The Tur\'an number, denoted by $ex(n, \mathscr{F})$, is the maximum number of edges in an $n$-vertex $\mathscr{F}$-free graph.
 Let \( F \) be a fixed graph with \( \chi(F) \geq 3 \). A forest $H$ is called a linear forest if all components of $H$ are
  paths. In this paper, we determined the exact value of \( ex(n, \{H, F\}) \) for a fixed graph $F$ with $\chi(F)\geq 3$ and a linear forest $H$ with at least $2$ components and each component with size at least $3$.
\end{abstract}


{\bf Keywords:} Tur\'an number; path; linear forest; chromatic number
\vskip.3cm

\section{Introduction}
In this paper, we only consider finite, simple and undirected graphs. Let $G=(V,E)$ be a graph, where $V$ is the vertex set and $E$ is the edge set of $G$. We use $v(G)$ and $e(G)$ to denote the order and size of $G$, respectively. Let $G$ and $H$ be two graphs. We use $G\cup H$ to denote the vertex disjoint union of  $G$ and $H$ and $G+H$ is the graph obtained from $G\cup H$ by adding all edges between $V(G)$ and $V(H)$. For a positive integer $k$, let $k H$ be the disjoint union of $k$ copies of  $H$. Let $I_{k}$ be an independent  set with $k$ vertices, $M_{s}$  a matching  of size $s$, $M_{k}^{t}$  a matching of size $t$ on $k$ vertices, $P_\ell$ a path with $\ell$ vertices and $K_m$  a complete graph with $m$ vertices. Let $A\subseteq V(G)$. Then $G[A]$ denotes the subgraph induced by $A$. For  $v\in V(G)\setminus A$, let $N_A(v)=\{u\in A:uv\in E(G)\}$ and $d_A(v)=|N_A(v)|.$ Let $S\subseteq V(G)$. Denote $N(S)=(\cup_{v\in S}N(v))\setminus S$.
Let $E'\subseteq E(G)$. We call $E'$ the {\em edge control set} of $G$ if for every $e\in E(G)\setminus E'$, there is $e'\in E'$ such that $e$ is incident with $e'$. The {\em edge control number} of $G$, denoted by $\beta_1(G)$, is the minimum size of edge control sets of $G$. We will use $\nu(G)$ to denote the matching number of $G$.

Let $\mathscr{F}$ be a family of graphs. A graph $G$ is $\mathscr{F}$-free if  $G$ does not contain any $F\in \mathscr{F}$  as a subgraph. The Tur\'an number, denoted by $ex(n, \mathscr{F})$, is the maximum number of edges in an $n$-vertex $\mathscr{F}$-free graph. If $\mathscr{F}=\{F\}$, we write $ex(n, F)$ instead of  $ex(n, \mathscr{F})$.
One of the earliest results concerning Tur\'an number is due to Tur\'an \cite{turan1954}, who showed that $ex(n, K_{k+1} ) = e (T (n, k))$, where the Tur\'an graph $T (n, k)$ is the complete $k$-partite $n$-vertex graph with each part of order $\lc\frac{n}{k}\rc$ or $\lf\frac{n}{k}\rf$.  When $k\geq n$, we will set $e (T (n, k))={n\choose 2}$. Simonovits \cite{simonovits1968method} showed that for large $n$, the extremal graph forbidding $p K_r$ is $K_{p-1}+T_{r-1}(n-p+1)$. Gorgol \cite{gorgol2011turan} initiated the Tur\'an numbers of disjoint copies of connected graphs and proved the following result.
\begin{theorem}[\cite{gorgol2011turan}]
    Let $G$ be an arbitrary connected graph on $n$ vertices and $p$ be an arbitrary positive integer. Then $ex(m,pG)\leq ex(m-(p-1)n,G)+{(p-1)n\choose 2}+(p-1)n(m-(p-1)n)$ for $m\geq pn$.
\end{theorem}
In 1959, Erd\H{o}s and Gallai \cite{gallai1959maximal} studied the value of $ex(n,P_\ell)$ and gave the following well known result.
\begin{theorem}[\cite{gallai1959maximal}]\label{G}
    For any $n,\ell \in \mathbb{N}$, $ex(n,P_\ell)\leq \frac{\ell-2}{2}n.$
\end{theorem}
Bushaw and Kettle \cite{bushaw2011turan} determined the Tur\'an number for the disjoint path with the same length when $n$ is large enough.
\begin{theorem}
[\cite{bushaw2011turan}]\label{thm: ex(n,kPl)}
For $k\geq 2$, $\ell \geq 4$ and $n\geq 2\ell+2k\ell(\lc\frac{\ell}{2}\rc+1){\ell\choose \lchuto},$
$$ex(n,k P_\ell)=\left(n-k\lchuto+1\right)\left(k\lchuto-1\right)+{k\lchuto-1\choose 2}+c_\ell,$$
where $c_\ell=1$ if $\ell$ is odd and $c_\ell=0$ if $\ell$ is even.
\end{theorem}
The extremal graph here is $K_t+I_{n-t}$ with a single edge added to the empty class when $\ell $ is odd and $t=k\lchuto-1$.

\begin{theorem}
[\cite{bushaw2011turan}]\label{thm: ex(n,kP3)}
For $k\geq2$, $n\geq 7k$,
$$ex(n,kP_3)={k-1\choose 2}+(n-k+1)(k-1)+\left\lfloor\frac{n-k+1}{2}\right\rfloor.$$
\end{theorem}
The extremal graph here is $K_{k-1}+M^{\lfloor\frac{n-k+1}{2}\rfloor}_{n-k+1}$.

 Lidick\'y, Liu  and  Palmer \cite{lidicky2012turan} determined the Tur\'an number for a linear forest, that is, a forest whose components are all paths.
\begin{theorem}[\cite{lidicky2012turan}]\label{thm: old linear forest}
    Let $F_k=\cup_{i=1}^kP_{\ell_i}$ and we set $\ell=\sum_{i=1}^k \lf\frac{\ell_i}{2}\rf$. If at least one $\ell_i$ is not 3, then for sufficiently large $n$,
    $$ex(n,F_k)=(\ell-1)(n-\ell+1)+{\ell-1\choose 2}+c,$$
    where $c=1$ if all $\ell_i$ are odd and $c=0$ otherwise.
\end{theorem}
For the linear forest that every component is $P_3$, Yuan and Zhang \cite{yuan2017turan} determined $ex(n,k P_3)$ for all positive integers $n$ and $k$. Let $F=\cup_{i=1}^kP_{\ell_i}$.
Yuan and Zhang \cite{yuan2021turan} improved the above results by determining  $ex(n,F_k)$ when $n\geq \sum_{i=1}^k\ell_i$ for all $\ell_1\geq \ell_2\geq\dots\geq \ell_k\geq 2$ and $k\geq 2$.
They also proposed a conjecture for the value of $ex(n,F_k)$ for all integer $n$ and $\ell_1\geq \ell_2\geq\dots\geq \ell_k\geq 2$, $\ell_1>3$. The conjecture is still basically open, but it was proven to be correct for some special linear forests. For example, $P_3\cup P_{2\ell+1}$ \cite{yuan2021turan}, $2P_5$ \cite{bielak2016turan,yuan2021turan}, $P_9\cup P_7$ \cite{fang2024turan}, $2P_7$ \cite{lan2019turan}, $2P_9$ \cite{zhang2022turan}, $3P_5$ \cite{feng2020turan}, $3P_7$ and $2P_3\cup P_{2\ell+1}$ \cite{deng2023turan} et al..
The Tur\'an number of star forest, a forest whose components are all stars, has been studied in \cite{lidicky2012turan,lan2019turan,Yin2016241}.

Recently, Alon and Frankl \cite{alon2024turan} proposed to consider $ex({n,\{F,M_{s+1}\}})$. In the same paper, they determined the value of $ex(n,\{K_{k+1},M_{s+1}\}).$
Later, Gerbner \cite{gerbner2024turan} gave several results about $ex({n,\{F,M_{s+1}\}})$,  when $F$ satisfies some properties. He proved that if $\chi(F)\geq 3$ and $n$ is large enough, then $ex(n,\{F,M_{s+1}\})=ex(s,\mathscr{F})+s(n-2)$, where $\mathscr{F}$ is the family of graphs obtained by deleting an independent set from $F$ and $\chi(F)$ is the chromatic number of $G$. Lu, Luo, and Zhao \cite{luo2024turannumbercompletebipartite} determined the exact value of $ex(n,\{K_{l,t},M_{s+1}\})$ when $s$ is large enough and $n\geq {3s\choose 2}$ for all $3\leq l\leq t$.

Since $M_{s+1}$ is a 1-regular graph with size  $s+1$, we can extend to study the Tur\'an number of any 2-regular graphs, i.e. cycles with length at least $k$ (denoted by $C_{\geq k}$) and a fixed graph $F$. For example, Dou, Ning and Peng \cite{dou2024number} determined the value of $ex(n,\{K_m, C_{\geq k}\}).$ If we consider $M_{s+1}$ as $s+1$ disjoint stars with maximum degree 1, we can extend to consider the Tur\'an number of $t+1$ disjoint star with maximum degree $\ell$ and a fixed graph, for example, Lu, Liu and Kang \cite{lu2024extremal} determined the value of $ex(n,\{K_m,(t+1) S_\ell\}\}).$

In this paper, we first consider the Tur\'an number of $kP_\ell$ and a fixed graph $F$ with $\chi(F)\geq 3$ and $\beta_1(F)\geq 1$, where $\chi(F)$ is the chromatic number of $F$ and $\beta_1(F)$ is the edge control number of $F$. Notice that  $\beta_1(F)\geq 1$ if $E(F)\not=\emptyset$. And $\beta_1(F)\geq 2$ implies the graph $K_2+I_{v(F)}$ is $F$-free.
Let
$$\mathscr{G}_1(F)=\{F[V(F)\backslash S]\mid S\subseteq V(F), e(F[S])=0\},$$and
$$\mathscr{G}_2(F)=\{F[V(F)\backslash S]\mid S\subseteq V(F),e(F[S])\leq 1\}.$$
Obviously, $\mathscr{G}_1(F)\subseteq \mathscr{G}_2(F)$. So $ex(k\lchuto-1,\mathscr{G}_2(F))\leq ex(k\lchuto-1,\mathscr{G}_1(F))$.
We obtain the following main results. We deal with the case when $\beta_1(F)\geq 2$ first.
\begin{theorem}\label{thm: ex(n,kpl,F)} Let $F$ be a graph of order $r$ with $\chi(F)\geq 3$ and $\beta_1(F)\geq 2$.
    When $k\geq 2$, $l\geq 4$, $r\geq 3$ and $n\geq 4k^2\ell^2(k\ell+r)(\lchuto+1){\ell\choose \lchuto} $, we have
    $$ex(n,\{F,k P_\ell\})=\left(n-k\lchuto+1\right)\left(k\lchuto-1\right)+b_\ell,$$ where  $b_\ell=\max\{1+ex(k\lchuto-1,\mathscr{G}_2(F)),ex(k\lchuto-1,\mathscr{G}_1(F))\} $ when $\ell$ is odd and $b_\ell= ex(k\lchuto-1,\mathscr{G}_1(F))$ when $\ell$ is even.
\end{theorem}
Now we consider the case $\beta_1(F)=1$.
Let $B_t=K_2+I_t$, a graph constructed by $t$ triangles sharing one edge. Then $v(B_t)=t+2$ and $e(B_t)=2t+1$.

Let $\beta_1(F)=1$ and $\chi(F)\geq 3$.  Let $N(F,K_3)$ denote the number of copies of $K_3$ in $F$. Then $N(F,K_3)\geq 1$. It is easy to check $F$ is a subgraph of $B_t$ for some $t\geq v(F)-2$. Set $f(F):=\min\{d_F(u_1),d_F(u_2)\}$, where $\{u_1u_2\}$ is the edge control set of $F$. We say a triplet $(F,k,\ell)$ has property $\mathscr{P}$ if one of the following conditions holds:

\noindent(1) $\ell$ is even;

\noindent(2) $v(F)=r\leq k\lchuto+1$;

\noindent(3) $N(F,K_3)=1$ and $f(F)\leq k\lchuto$.

\noindent Assume $F$ has isolate vertices. Then we denote the graph by deleting all isolate vertices from $F$ by  $F'$. If $F'$ is a subgraph  of $ G$, then $F$ is a subgraph  of $ G$ when $n\geq v(F)$. Thus, for convenience, we only consider the case that $F$ has no isolate vertices when $\beta_1(F)=1$. We have the following result.
\begin{theorem}\label{thm: when beta F=1}
    Let $F$ be a graph of order $r$ with $\chi(F)\geq 3$, $\beta_1(F)=1$ and without isolate vertices.
    When $k\geq 2$, $l\geq 4$, $r\geq 3$ and $n\geq 4k^2\ell^2(k\ell+r)(\lchuto+1){\ell\choose \lchuto} $, we have
    $$ex(n,\{F,k P_\ell\})=\left(n-k\lchuto+1\right)\left(k\lchuto-1\right)+C_{F,k,\ell},$$where $C_{F,k,\ell}=0$ if $(F,k,\ell)$ has property $\mathscr{P}$ and $C_{F,k,\ell}=1$ otherwise.
\end{theorem}
When $F=K_r$, we easily have the following corollary.
\begin{cor}
   Let $k\geq 2$, $\ell\geq 4$ and $n\geq 4k^2\ell^2(k\ell+r)(\lchuto+1){\ell\choose \lchuto}$.

   \noindent  (i) We have
    $$ex(n,\{K_3,kP_\ell\})=\left(n-k\lchuto+1\right)\left(k\lchuto-1\right).$$
     (ii) If $4\leq r\leq k\lchuto+1$, then we have
    $$ex(n,\{K_r,kP_\ell\})=\left(n-k\lchuto+1\right)\left(k\lchuto-1\right)+e\left(T\left(k\lchuto-1,r-2\right)\right).$$
    (iii) If $r\geq k\lchuto+2$, then we have
    $$ex(n,\{K_r,kP_\ell\})=\left(n-k\lchuto+1\right)\left(k\lchuto-1\right)+{k\lchuto-1\choose 2}+D_\ell,$$
    where $D_\ell=1$ if $\ell$ is odd, and $D_\ell=0$ if $\ell$ is even.
\end{cor}
\begin{proof}
     $(i)$ Since $\beta_1(K_3)=1$, $\chi(K_3)=3$ and $(K_3,k,\ell)$ has property $\mathscr{P}$, the result holds from Theorem \ref{thm: when beta F=1}.

     Note that $\beta_1(K_r)\ge2$ and $\chi(K_r)\ge3$  when  $r\geq 4$. We have $ex(n,\mathscr{G}_1(K_r))=ex(n,K_{r-1})=e(T(n,r-2))$ and $ex(n,\mathscr{G}_2(K_r))=ex(n,K_{r-2})=e(T(n,r-3))$. It is easy to check when $r\leq n+2$, $ex(n,K_{r-1})\geq ex(n,K_{r-2})+1$.

     $(ii)$ If $4\leq r\leq k\lchuto+1$, $ex(k\lchuto-1,\mathscr{G}_2(K_r))+1\leq ex(k\lchuto-1,\mathscr{G}_1(K_r))$. So $b_\ell=ex(k\lchuto-1,\mathscr{G}_1(K_r))$ regardless the parity of $\ell$ and the result holds by Theorem \ref{thm: ex(n,kpl,F)}.

   $(iii)$ If $r\geq k\lchuto+2$, $ex(k\lchuto-1,\mathscr{G}_2(K_r))=ex(k\lchuto-1,\mathscr{G}_1(K_r))={k\lchuto-1\choose 2}$. Thus $b_\ell={k\lchuto-1\choose 2}+1$ when $\ell$ is odd, and $b_\ell={k\lchuto-1\choose 2}$ when $\ell$ is even. Thus the result holds by Theorem \ref{thm: ex(n,kpl,F)}.
\end{proof}
\vskip.2cm
Let $\sigma(F)$ denote the maximum positive integer such that $I_{v(F)}+M_{\sigma(F)}$ is $F$-free. Note that $\sigma(F)$ may be infinite. For example,   $I_{n_1}+M_{n_2}$ is $F$-free for all $n_1,n_2\geq 1$ if $K_4\subseteq F$ which implies $\sigma(F)=\infty$.
Assume $\sigma(F) < \infty$. By the definition of $\sigma(F)$, there is a copy of $F$ in $I_{v(F)}+M_{\sigma(F)+1}$ and for every $e\in M_{\sigma(F)+1}$, $e\cap V(F)\neq \emptyset$. So $\sigma(F) \leq v(F)$. Define a family of graphs $\mathscr{H}_i(F)$ ( $i \geq 0$), as follows:
\begin{align*}
    \mathscr{H}_i(F) = \{F[V(F)\backslash S] : S \subseteq V(F), \ \Delta(F[S]) \le 1, \ e(F[S]) \leq i\}.
\end{align*}
Note that $\mathscr{H}_0(F)$ and $\mathscr{H}_1(F)$ correspond to $\mathscr{G}_1(F)$ and $\mathscr{G}_2(F)$, respectively. We have $\mathscr{H}_i(F) \subseteq \mathscr{H}_{i+1}(F)$. For sufficiently large $i_0$ ($i_0 \geq \frac{v(F)}{2}$), $\mathscr{H}_i(F)$ becomes fixed for $i \geq i_0$. Let $\mathscr{H}(F) = \bigcup_{i=1}^{\infty} \mathscr{H}_i(F) = \bigcup_{i=1}^{i_0} \mathscr{H}_i(F)$, i.e.
\begin{align*}
    \mathscr{H}(F) = \{F[V(F)\setminus S] : S \subseteq V(F), \ \Delta(F[S]) \le 1\}.
\end{align*}
The following theorem is another main result of the paper.
\begin{theorem}\label{p3} Let $F$ be a graph with $\chi(F) \geq 3$.
   Let $k \geq 2$ and $n \geq 9(k^2+k+1)+2v(F)$.  If $\sigma(F) = \infty$, then we have
    \begin{align*}
        ex(n,\{k P_3,F\}) = (k-1)(n-k+1) + \left\lfloor \frac{n-k+1}{2} \right\rfloor + ex(k-1, \mathscr{H}(F)).
    \end{align*}
If $\sigma(F) < \infty$, then we have
    \begin{align*}
        ex(n,\{k P_3,F\}) = \max_{0 \leq i \leq \sigma(F)} \left\{(k-1)(n-k+1) + i + ex(k-1, \mathscr{H}_i(F))\right\}.
    \end{align*}
\end{theorem}
When $F=K_r$,  we have the following result.
\begin{cor}
    For $k \geq 2$ and $n \geq 9(k^2+k+1)+2r$,  we have
 $$ex(n,\{k P_3,K_3\}) = (k-1)(n-k+1)$$
    and for $r\ge 4$,
    $$ex(n,\{k P_3,K_r\}) =(k-1)(n-k+1) + \left\lfloor \frac{n-k+1}{2} \right\rfloor +e(T(k-1,r-3)).$$
\end{cor}
\begin{proof}
    When $r=3$, we have $\sigma(K_3)=0$ and $\mathscr{H}_0(K_3)=\{K_2\}$. Thus $ex(k-1,\mathscr{H}_0(K_3))=0$. By Theorem \ref{thm: ex(n,kP3)}, we have the exact result of $ex(n,\{k P_3,K_3\})$.

    When $r\geq 4$,  $\sigma(K_r)=\infty$ and $\mathscr{H}(K_r)=\{K_{r-1},K_{r-2}\}$. By Theorem \ref{thm: ex(n,kP3)}, we have
    $$ex(n,\{k P_3,K_r\})=(k-1)(n-k+1)+ \left\lfloor \frac{n-k+1}{2} \right\rfloor +ex(k-1,K_{r-2}),$$
    and we are done.
\end{proof}

A forest is called \textbf{linear forest} if every component is a path.
 Suppose $H$ is a linear forest with $k$ components and let $I_1$ (resp. $I_2$) be the collection of indices of even paths (resp. odd paths). Then we have the following result.
\begin{theorem}\label{thm: linear forest and beta 2}
     Let  $H=(\cup_{i\in I_1}P_{2\ell_i})\cup (\cup_{j\in I_2}P_{2\ell_j+1})$ with $|I_1|+|I_2|=k$ and $\ell=\sum_{i=1}^k\ell_i$, where  $\ell_i\geq 2$ when $i\in I_1$ and $\ell_j\geq 1$ when $j\in I_2$. Let $F$ be a graph with $\chi(F)\geq 3$ and $\beta_1(F)\geq 2$. Assume  $H\neq kP_3$ and $n$ is large enough.

    (i) If  $I_1\not=\emptyset$, then $$ex(n,\{H,F\})=(\ell-1)(n-\ell+1)+ex(\ell-1,\mathcal{G}_1(F)).$$

    (ii)    If $I_1=\emptyset$, then $$ex(n,\{H,F\})=(\ell-1)(n-\ell+1)+\max{\{ex(\ell-1,\mathcal{G}_1(F)},ex(\ell-1,\mathcal{G}_2(F))+1\}.$$
 \end{theorem}
Let $H=(\cup_{i\in I_1}P_{2\ell_i})\cup (\cup_{j\in I_2}P_{2\ell_j+1})$, where  $\ell_i\geq 2$ when $i\in I_1$ and $\ell_j\geq 1$ when $j\in I_2$. Assume $|I_1\cup I_2|=k$ and $H\neq kP_3$. For a fixed graph $F$ with $\chi(F)\geq 3$ and $\beta_1(F)=1$, we say the pair $(F,H)$ has property $\mathscr{R}$ if one of the following conditions holds:

(1)  $I_1\neq \emptyset$;

(2) $v(F)=r \leq \sum_{i=1}^k \ell_i+1$;

(3) $N(F,K_3)=1$ and $f(F)\leq \sum_{i=1}^k \ell_i$.

Set $D_{F,H}=0$ if $(F,H)$ has property $\mathscr{R}$, and $D_{F,H}=1$ otherwise.
Then we have the result when $\beta_1(F)=1$.
\begin{theorem}\label{thm: linear forest and beta 1}
     Let  $H=(\cup_{i\in I_1}P_{2\ell_i})\cup (\cup_{j\in I_2}P_{2\ell_j+1})$, where $|I_1\cup I_2|=k\ge 2$,  $\ell_i\geq 2$ when $i\in I_1$ and $\ell_j\geq 1$ when $j\in I_2$. Assume $\ell=\sum_{i=1}^k\ell_i$ and $H\neq kP_3$. Let $F$ be a graph with $\chi(F)\geq 3$ and $\beta_1(F)=1$. When $n$ is large enough, we have
    $$ex(n,\{H,F\})=(\ell-1)(n-\ell+1)+D_{F,H}.$$
 \end{theorem}
When $F=K_r$, we easily have the following corollary.
\begin{cor}Let
    $H=(\cup_{i\in I_1}P_{2\ell_i})\cup (\cup_{j\in I_2}P_{2\ell_j+1})$, where $|I_1\cup I_2|=k\ge 2$,  $\ell_i\geq 2$ when $i\in I_1$ and $\ell_j\geq 1$ when $j\in I_2$.  Assume $\ell=\sum_{i=1}^k\ell_i$ and $H\neq kP_3$. When $n$ is large enough, we have the following results.

   \noindent (i) We have
    $$ex(n,\{K_3,H\})=\left(n-\ell+1\right)\left(\ell-1\right).$$
    (ii) If $4\leq r\leq \ell+1$, then we have
    $$ex(n,\{K_r,H\})=\left(n-\ell+1\right)\left(\ell-1\right)+e\left(T\left(\ell-1,r-2\right)\right).$$
    (iii) If $r\geq \ell+2$, then we have
    $$ex(n,\{K_r,H\})=\left(n-\ell+1\right)\left(\ell-1\right)+{\ell-1\choose 2}+D_H,$$
    where $D_H=1$ if $I_1=\emptyset$ and $D_H=0$ otherwise.
\end{cor}
\begin{proof}
    $(i)$ Since $\beta_1(K_3)=1$, $\chi(K_3)=3$ and $(K_3,H)$ has property $\mathscr{R}$, the result holds from Theorem \ref{thm: linear forest and beta 1}.

     Note that $\beta_1(K_r)\ge2$ and $\chi(K_r)\ge3$  when  $r\geq 4$. We have $ex(n,\mathscr{G}_1(K_r))=ex(n,K_{r-1})=e(T(n,r-2))$ and $ex(n,\mathscr{G}_2(K_r))=ex(n,K_{r-2})=e(T(n,r-3))$.
     It is easy to check when $r\leq n+2$, $ex(n,K_{r-1})\geq ex(n,K_{r-2})+1$.

     $(ii)$  If $4\leq r\leq \ell+1$, $ex(\ell-1,\mathscr{G}_2(K_r))+1\leq ex(\ell-1,\mathscr{G}_1(K_r))$. So the result holds by Theorem \ref{thm: linear forest and beta 2}.

   $(iii)$ If $r\geq \ell+2$, $ex(\ell-1,\mathscr{G}_2(K_r))=ex(\ell-1,\mathscr{G}_1(K_r))={\ell-1\choose 2}$.  Thus the result holds by Theorem \ref{thm: linear forest and beta 2}.
\end{proof}

This paper is organized as follows. In section \ref{sec:kpl}, we will prove Theorems \ref{thm: ex(n,kpl,F)} and \ref{thm: when beta F=1}. In section 3, we will prove Theorem \ref{p3}. In section 4, we will prove Theorems \ref{thm: linear forest and beta 2} and \ref{thm: linear forest and beta 1}.
\section{Proofs of Theorems \ref{thm: ex(n,kpl,F)} and \ref{thm: when beta F=1}}\label{sec:kpl}

In this section, let $F$ be a graph of order $r$ with $\chi(F)\geq 3$ and we consider $ex(n,\{k P_\ell,F\})$, where $\ell\ge 4$ and $k\ge 2$.

\subsection{The low bounds}
In this subsection, we give the lower bound of $ex(n,\{k P_\ell,F\})$, where $\ell\ge 4$ and $k\ge 2$.

We first consider the case $\beta_1(F)\geq 2$.
Let $H_1(n,k,\ell,F)$ be the graph obtained by embedding a $\mathscr{G}_1(F)$-free graph of size $ex(k\lchuto-1,\mathscr{G}_1(F))$ to the smaller part of $K_{k\lchuto-1,n-k\lchuto+1}$. Let $H_2(n,k,\ell,F)$ be the graph obtained by embedding a $\mathscr{G}_2(F)$-free graph of size $ex(k\lchuto-1,\mathscr{G}_2(F))$  to the smaller part of $K_{k\lchuto-1,n-k\lchuto+1}$ and adding one edge in the larger part. It is easy to check $H_1(n,k,\ell,F)$ is $\{k P_\ell,F\}$-free and $H_2(n,k,\ell,F)$ is $\{k P_\ell,F\}$-free if $\ell$ is odd. Hence we have $$ex(n,\{F,k P_\ell\})\ge \left(n-k\lchuto+1\right)\left(k\lchuto-1\right)+b_\ell,$$where  $b_\ell=\max\{1+ex(k\lchuto-1,\mathscr{G}_2(F)),ex(k\lchuto-1,\mathscr{G}_1(F))\} $ when $\ell$ is odd and $b_\ell= ex(k\lchuto-1,\mathscr{G}_1(F))$ when $\ell$ is even.

Now we consider the case $\beta_1(F)=1$.
Note that $K_{k\lchuto-1,n-k\lchuto+1}$ is $\{k P_\ell,F\}$-free.
Then  $K_{k\lchuto-1,n-k\lchuto+1}$ provides the lower bound of $ex(n,\{k P_\ell,F\})$ when $(F,k,\ell)$ has condition $\mathscr{P}$. So  $ex(n,\{k P_\ell,F\})\ge (k\lchuto-1)(n-k\lchuto+1)$ if $\beta_1(F)=1$ and $(F,k,\ell)$ has condition $\mathscr{P}$.

Let $\beta_1(F)=1$ and $(F,k,\ell)$ do not have condition $\mathscr{P}$. Then $\ell$ is odd, $r=v(F)>k\lchuto+1$, and $N(F,K_3)\ge 2$ or $f(F)> k\lchuto$. Adding one edge to the larger part of $K_{k\lchuto-1,n-k\lchuto+1}$ and we denote the obtained graph by $G_0$. Then $e(G_0)=(k\lchuto-1)(n-k\lchuto+1)+1$. It is easy to check that  $G_0$ is $kP_\ell$-free by $\ell$ being odd.
We will prove it is $F$-free. Let $\{v_1v_2\}$ be the edge control set of $F$.

If $N(F,K_3)\geq 2$, then $v_1v_2$ is contained in at least two copies of $K_3$. If $G_0$ has a copy of $F$, then  $v_1v_2$ must be the unique edge  in the larger part. Since $|N_{G_0}(v_1)\cup N_{G_0}(v_2)|=k\lchuto-1+2<r$, we have a contradiction with $|N_F(v_1)\cup N_F(v_2)|=r$.
If $N(F,K_3)=1$, then  $f(F)\geq k\lchuto+1$. In this case, $v_1v_2$ is contained in the $K_3$. If $v_1v_2$ is the unique edge in the larger part, then we can similarly have a contradiction. If  $v_1$ (resp. $v_2$) is in the smaller part (resp.  the larger part) of $G_0$, then $\min\{d_{G_0}(v_1),d_{G_0}(v_2)\}=d_{G_0}(v_2)\leq k\lchuto$, but $f(F)=\min\{d_F(v_1),d_F(v_2)\}\geq k\lchuto+1$, a contradiction. Thus $G_0$ is $F$-free. So  $ex(n,\{k P_\ell,F\})\ge (k\lchuto-1)(n-k\lchuto+1)+1$ if $\beta_1(F)=1$ and $(F,k,\ell)$ has no condition $\mathscr{P}$.

\subsection{The upper bounds}

Now we consider the upper bound of $ex(n,\{k P_\ell,F\})$. In the following discussion, we assume that $G$ is a $\{k P_\ell,F\}$-free graph of order $n$ with $ex(n,\{k P_\ell,F\})$ edges, where $\ell\ge 4$, $k\ge 2$ and $n\geq  4k^2\ell^2(k\ell+r)(\lchuto+1){\ell\choose \lchuto} $. By the low bound of $ex(n,\{k P_\ell,F\})$,
\begin{equation}\label{eq-0}
e(G)\ge \left(k\lchuto-1\right)\left(n-k\lchuto+1\right).
\end{equation}

We first give some lemmas which will be used in the proof of Theorem \ref{thm: ex(n,kpl,F)}.
\begin{lemma}[\cite{bushaw2011turan}]\label{lem: t vertices with many common neighbors}
    Let $H$ be a graph on $n$ vertices with $m$ edges, $t\in \mathbb{N}$, and let $F_1,F_2$ be arbitrary graphs. If $F_1\cup F_2\not\subseteq H$, then any $F_1$ in $H$ contains $t$ vertices with shared neighborhood of size at least $ \frac{m'-(n-r)(t-1)}{r-t+1}/{r\choose t}$, where $m'=m-ex(n-r,F_2)-{r\choose 2}$, and $r=|V(F_1)|$.
\end{lemma}
 Since $G$ is $kP_\ell$-free and $k\ge 2$, $P_\ell\cup (k-1)P_\ell \not\subseteq G$. According to Lemma  \ref{lem: t vertices with many common neighbors}, for every copy of $P_\ell$ in $G$, there are $\lf\frac{\ell}{2}\rf$ vertices in the $P_\ell$ with shared neighborhood of size at least
\begin{equation*}
    \begin{aligned}
        n'=& \frac{(k\lchuto-1)(n-k\lchuto+1) -ex(n-\ell,(k-1) P_\ell)-{\ell \choose 2}-(n-\ell)(\lchuto -1)}{(\lchutoo+1){\ell \choose \lchuto}}\\
        \geq & \frac{n/3-{\ell\choose 2}}{(\lchutoo+1){\ell \choose \lchuto}}+k\ell.
    \end{aligned}
\end{equation*}
The  inequality holds by Theorems \ref{G},  \ref{thm: ex(n,kPl)} and  $n\geq  4k^2\ell^2(k\ell+r)(\lchuto+1){\ell\choose \lchuto} $.
Then  there is $S\subseteq V(G)$ such that $|S|=\lchuto$ and $|N(S)|\ge n_k$, where $n_k=\frac{n/3-{\ell\choose 2}}{(\lchutoo+1){\ell \choose \lchuto}}+k\ell>k\ell$.

We now construct an auxiliary $\lchuto$-uniform hypergraph $\mathscr{H}$ as follows: $$E(\mathscr{H})=\left\{S\subseteq V(G):|S|=\lchuto,|N(S)|\ge n_k\right\}$$ and $V(\mathscr{H})=\cup_{S\in E(\mathscr{H})}S$. For short, let $V(\mathscr{H})=A$. Then for any copy of $P_\ell$ in $G$, there is $S\in E(\mathscr{H})$ such that $S\subseteq V(P_\ell)$. On the other hand, for any $S\in E(\mathscr{H})$, there is a copy of $P_\ell$ in $G$ such that $S\subseteq V(P_\ell)$. Since $n_k>k\ell$ and $G$ is $kP_\ell$-free,  $\nu(\mathscr{H})\le k-1$. We construct a graph $G_{\mathscr{H}}$ associated with $\mathscr{H}$, where $V(G_{\mathscr{H}})=A$ and for any $u,v\in V(G_{\mathscr{H}})$, $uv\in E(G_{\mathscr{H}})$ if and only if there is $S\in E(\mathscr{H})$ such that $u,v\in S$. By the construction of $G_{\mathscr{H}}$, if there is a copy of $P_{\lchuto}$ in $G_{\mathscr{H}}$, then there is a copy of $P_\ell$ in $G$. So $G_{\mathscr{H}}$ is $kP_{\lchuto}$-free by $G$ being $kP_\ell$-free.

Let $C\subseteq V(\mathscr{H})$, we denote $\mathscr{H}[C]$ as the sub-hypergraph of $\mathscr{H}$  with vertex set $C$ and  hyperedges  $\{E\in E(\mathscr{H}): E\subseteq C\}$. Then we have the following lemma.
 \begin{lemma}\label{lem: delete t vertices always have k-1 edges}
    For every $X\subseteq A$ with $|X|=t< \lchuto$, $\nu(\mathscr{H}[A\setminus X])\ge k-1$.
\end{lemma}
\begin{proof}
   By contradiction. Suppose there is $X\subseteq A$ with $|X|=t< \lchuto$ such that $\nu(\mathscr{H}[A\setminus X])<k-1$. Then $G[V(G)\setminus X]$ is  $(k-1)P_\ell$-free. When $k\geq 3$, by Theorem  \ref{thm: ex(n,kPl)}, we have $e(G[V(G)\setminus X])\leq (n-t-(k-1)\lchuto+1)((k-1)\lchuto-1)+{(k-1)\lchuto-1\choose 2}+c_\ell$. Then
    \begin{equation*}
    \begin{aligned}
                e(G)\leq& e(G[V(G)\setminus X])+ t(n-t)+{t\choose 2}\\
        <&\left(n-k\lchuto+1\right)\left(k\lchuto-1\right),
    \end{aligned}
    \end{equation*}
    a contradiction with (\ref{eq-0}).
    When $k=2$, by Theorem \ref{G}, we have $e(G[V(G)\setminus X])\leq \frac{\ell-2}{2}(n-t)$, then
    \begin{equation*}
    \begin{aligned}
            e(G)\leq& e(G[V(G)\setminus X])+ t(n-t)+{t\choose 2}\\
        <&\left(n-2\lchuto+1\right)\left(2\lchuto-1\right),
    \end{aligned}
    \end{equation*}
    a contradiction with (\ref{eq-0}).
\end{proof}

By Lemma \ref{lem: delete t vertices always have k-1 edges}, $\nu(\mathscr{H})= k-1$. Since $n_k>k\ell$ and $G$ is $kP_\ell$-free,  $\nu(\mathscr{H}[A\setminus S])\le k-2$ for any $S\in E(\mathscr{H})$ which implies $G[V(G)\setminus S]$ is $(k-1)P_\ell$-free.
Denote
\begin{equation*}
    \begin{aligned}
        B=&\left\{x\in V(G)\setminus A: d_{A}(x)\geq (k-1)\lchuto\right\},\\
        C=& \left\{x\in V(G)\setminus A: 0<d_{A}(x)< (k-1)\lchuto\right\},\\
        D=&\{x\in V(G)\setminus A: d_{A}(x)=0\}.\
    \end{aligned}
\end{equation*} Then $V(G)=A\cup B\cup C\cup D$. Since $\nu(\mathscr{H})= k-1$,
$D$ is $P_\ell$-free. Then we have $e(G[D])\leq \frac{l-2}{2}|D|$ by Lemma \ref{G}. By the definition of $\mathscr{H}$, we have the following result.
\begin{lemma}\label{lemma: edges inside BCD}
        (i). For any $u\in D$, $|N(u)\cap (B\cup C)|\le 1$ and $d_{G[B\cup C]}(u)\le 1$ for any $u\in B\cup C$.

    (ii). If $\ell$ is even, then $e(G[B])=0$. If $\ell$ is odd and $|A|<2(k-1)\lchuto$, then $e(G[B])\le 1$.
\end{lemma}
\begin{proof} ($i$) Suppose there is $u\in D$ such that $|N(u)\cap (B\cup C)|\ge 2$. Set $v,w\in N(u)\cap (B\cup C)$. Then there is $S\in E(\mathscr{H})$ such that $N(v)\cap S\not=\emptyset$. By the definition of $\mathscr{H}$ and $n_k>k\ell$, there is a $P_\ell=vx_1\ldots x_{\ell-1}$ such that $x_1,x_3,\ldots,x_{\ell-2}\in S$ (resp. $x_1,x_3,\ldots,x_{\ell-1}\in S$) when $\ell$ is odd (resp. $\ell$ is even) and $u,w\notin V(P_\ell)$. Now we have a $P'_\ell=wuvx_1\ldots x_{\ell-3}$. Let $X=\{x_1,x_3,\ldots,x_{\ell-2}\}$ (resp. $X=\{x_1,x_3,\ldots,x_{\ell-1}\}$) when $\ell$ is odd (resp. $\ell$ is even).  By Lemma   \ref{lem: delete t vertices always have k-1 edges}, $\nu(\mathscr{H}[A\setminus X])\ge k-1$. Since $n_k>k\ell$, there are $k-1$ disjoint copies of $P_{\ell}$ in $G[V\setminus V(P'_\ell)]$ which imply there are $k$ disjoint copies of $P_{\ell}$ in $G$, a contradiction. By the same argument, we have  $d_{G[B\cup C]}(u)\le 1$ for any $u\in B\cup C$.

 ($ii$) By the same argument as ($i$), we have $e(G[B])=0$ if $\ell$ is even.
Now we consider the case that $\ell$ is odd and $|A|<2(k-1)\lchuto$. Suppose $e(G[B])\ge 2$. Let  $u_1u_2,w_1w_2\in E(G[B])$ and $u_1\not=w_1$.
Since $u_1,w_1\in B$ and  $|A|<2(k-1)\lchuto$, there is $y\in A\cap N(u_1)\cap N(w_1)$. By Lemma  \ref{lem: delete t vertices always have k-1 edges}, $\nu(\mathscr{H}[A\setminus \{y\}])\ge k-1$.
Since $u_2\in B$ and  $|A|<2(k-1)\lchuto$, there is $S\in E(\mathscr{H}[A\setminus \{y\}])$ such that $N(u_2)\cap S\not=\emptyset$. By the definition of $\mathscr{H}$ and $n_k>k\ell$, there is a $P_\ell=u_2x_1\ldots x_{\ell-1}$ such that $x_1,x_3,\ldots,x_{\ell-2}\in S$. Now we have a $P'_\ell=w_1yu_1u_2x_1\ldots x_{\ell-4}$. By the same argument as ($i$), we can derive a contradiction.
\end{proof}

We will prove that the size of $A$ is bounded by $k\lchuto-1$. This has been proved in the case of $k=2$ in \cite{bushaw2011turan},
\begin{lemma}[\cite{bushaw2011turan}]\label{lem: the size of A when k=2}
   When $\ell\geq 4$ and $n\geq 2\ell+4\ell(\lc\frac{\ell}{2}\rc+1){\ell\choose \lchuto}$, let $G$ be a graph with $n$ vertices and without two disjoint $P_\ell$, and $e(G)\geq (2\lchuto-\frac{1}{2})n$,  then $|A|\leq 2\lchuto-1$.
\end{lemma}

\begin{lemma}\label{lem: the size of A} Let $\ell\ge 4$, $k\ge 2$ and $n\geq  4k^2\ell^2(k\ell+r)(\lchuto+1){\ell\choose \lchuto} $. If $G$ is a $kP_\ell$-free graph with $n$ vertices, and $e(G)\geq (k\lchuto-\frac{1}{2})n$,
  then $|A|\le k\lchuto -1$.
\end{lemma}
\begin{proof}
    When $k=2$, the result holds by Lemma \ref{lem: the size of A when k=2}.
   Let $k\ge 3$. Suppose  $|A|\ge k\lchuto$. Then $|E(\mathscr{H})|\ge k$. Let $S_0\in E(\mathscr{H})$ and $x\in S_0$. Set $S'=S_0\setminus\{x\}$. By Lemma \ref{lem: delete t vertices always have k-1 edges}, $\nu(\mathscr{H}[A\setminus  S'])\ge k-1$. Let $S_1,\ldots,S_{k-1}$ be the independent edges in $\mathscr{H}[A\setminus  S']$. Since $\nu(\mathscr{H})= k-1$, we have that $\{S_1,\ldots,S_{k-1}\}$ is the maximum matching in $\mathscr{H}$ and then we can assume $S_0\cap S_1=\{x\}$. Denote $S^x=(S_0\cup S_1)\setminus\{x\}$ and $\overline{A}=A\setminus (\cup_{i=0}^{k-1}S_i)$. Since $|A|\ge k\lchuto$, $\overline{A}\not=\emptyset$. For any $y\in \overline{A}$, denote $S_y\in E(\mathscr{H})$ with $y\in S_y$.

  \noindent{\bf Claim 1.} For any $S\in E(\mathscr{H})$, $|V(\mathscr{H}[A\setminus S])|\leq (k-1)\lchuto-1$.

\noindent{\bf Proof of Claim 1.} Suppose there is $S\in E(\mathscr{H})$ such that $|V(\mathscr{H}[A\setminus S])|\ge  (k-1)\lchuto$.  Let $G'=G[V(G)\setminus S]$. Since $G$ is $kP_\ell$-free and $|N(S)|\ge n_k>kl$, we have $G'$ is $(k-1)P_\ell$-free, and
\begin{equation*}
    \begin{aligned}
        e(G')\geq & e(G)-e(S,V(G)\setminus S)-e(G[S])\\
        \geq & \left(k\lchuto-\frac{1}{2}\right)n-\lchuto\left(n-\lchuto\right)-{\lchuto\choose 2}\\
        >&\left((k-1)\lchuto-\frac{1}{2}\right)\left(n-\lchuto\right).
    \end{aligned}
\end{equation*}
 Now we construct an auxiliary $\lchuto$-uniform hypergraph $\mathscr{H'}$  as follows: $E(\mathscr{H'})=\{S\subseteq V(G')||S|=\lchuto,|N_{G'}(S)|\ge n_{k-1}\}$ and $V(\mathscr{H'})=\cup_{S\in E(\mathscr{H'})}S$. Note that $|N_{G'}(S')|\ge n_k-\lchuto\geq n_{k-1}$ for any $S'\in E(\mathscr{H}[A\setminus S])$. Then $\mathscr{H}[A\setminus S]$ is a sub-hypergraph of $\mathscr{H'}$. Note that the order of $G'$ is
\begin{equation*}
    \begin{aligned}
n-\lchuto&\geq  4k^2\ell^2(k\ell+r)\left(\lchuto+1\right){\ell\choose \lchuto} -\lchuto\\
&\geq  4(k-1)^2\ell^2((k-1)\ell+r)\left(\lchuto+1\right){\ell\choose \lchuto}.
\end{aligned}
\end{equation*}
        By the induction hypothesis, $|V(\mathscr{H}[A\setminus S])|\le |V(\mathscr{H'})|\leq (k-1)\lchuto-1$, a contradiction.\q

\vskip.2cm

\noindent  {\bf Claim 2.} For any $y\in \overline{A}$, $S_y\cap S_j\not=\emptyset$ for all $2\le j\le k-1$ and $S_y\cap S^x=\emptyset$.

\noindent  {\bf Proof of Claim 2.} Suppose there is $y\in \overline{A}$ and $2\le j\le k-1$, say $j=k-1$, such that $S_y\cap S_{k-1}=\emptyset$. Then $S_y,S_j\in E(\mathscr{H}[A\setminus S_{k-1}])$ for all $0\le j\le k-2$. Thus $|V(\mathscr{H}[A\setminus S_{k-1}])|\ge (k-1)\lchuto$, a contradiction with Claim 1.

Suppose there is $y\in \overline{A}$ such that $S_y\cap S^x\not=\emptyset$, say $z\in S_y\cap S_0$ and $z\not=x$. Then there are two disjoint copies of $P_{\lchuto}$ in $G_{\mathscr{H}}[S_0\cup S_1\cup \{y\}]$. Note that there are $k-2$ disjoint copies of $P_{\lchuto}$ in $G_{\mathscr{H}}[\cup_{i=2}^{k-1}S_i]$. So there are $k$ disjoint copies of $P_{\lchuto}$ in $G_{\mathscr{H}}$, a contradiction.
\q
\vskip.2cm



    \noindent  {\bf  Claim 3.}\label{claim: E1E2 is connected}
        If there is $\overline{S}\in E(\mathscr{H})$ such that $\overline{S}\subseteq S^x$, then there exists $S'\in E(\mathscr{H})$ with $S'\cap S^x\neq \emptyset$ and $S'\not\subseteq S_0\cup S_1$.

   \noindent  {\bf Proof of Claim 3.} By contradiction. Suppose for any $S'\in E(\mathscr{H})$ with $S'\cap S^x\neq \emptyset$, we have  $S'\subseteq S_0\cup S_1$.

         First, we assume there is $S_x\in E(\mathscr{H}[A\setminus S^x])$ such that $x\in S_x$. Then $V(\mathscr{H}[A\setminus S^x])=A\setminus S^x$. Denote $A_1=V(\mathscr{H}[A\setminus S^x])$. Since $G$ is $kP_\ell$-free and $\overline{S}\in E(\mathscr{H})\setminus E(\mathscr{H}[A\setminus S^x])$,
          $G[V(G)\setminus S^x]$ is $(k-1)P_\ell$-free. By the same argument as that proof of
         Claim 1 and note that  $n_k-2\lchuto\geq n_{k-1}$, we have $|A_1|\leq (k-1)\lchuto-1$.
         Then $|A|=|A_1|+|S^x |\leq (k+1)\lchuto-3$. Since  $|A|\ge k\lchuto$, we have
         $|A_1|\geq (k-2)\lchuto+2$.
       Let  $$U=\left\{x\in V(G)\setminus A ~:~ d_A(x)\geq k\lchuto-1 \right\}.$$ Then $U\subseteq B$. Since $|A_1|\leq (k-1)\lchuto-1$, $d_{S^x}(u)\ge \lchuto $ for any  $u\in U$. We claim that $|U|\leq 2n_k{2\lchuto-2\choose \lchuto}$.

        Suppose $|U|> 2n_k{2\lchuto-2\choose \lchuto}$. Since $d_{S^x}(u)\ge \lchuto $ for any  $u\in U$, there is $A^*\subseteq S^x$ such that $|A^*|=\lchuto-1$ and $|N(A^*)\cap U|\ge 2n_k$.
        Let $U^*=N(A^*)\cap U$. Since $|A_1|\leq (k-1)\lchuto-1$, $d_{A_1\setminus \{x\}}(u)\ge k\lchuto-1-|S^x\cup\{x\}|=(k-2)\lchuto$ for any $u\in U^*$.  Since $|U^*|\geq 2n_k$ and $2n_k(k-2)\lchuto>(n_k-1)((k-1)\lchuto-2)\geq (n_k-1)|A_1\setminus\{x\}|$ when $k\geq 3$,  there exist a vertex $a^*\in A_1\setminus\{x\}$ such that $d_{U^*}(a^*)\ge  n_k$. Thus $|N(A^*\cup\{a^*\})|\ge n_k$ which implies  $A^*\cup\{a^*\}\in E(\mathscr{H})$. But $(A^*\cup\{a^*\})\cap S^x\neq \emptyset$ and $(A^*\cup\{a^*\})\not\subseteq S_0\cup S_1$, a contradiction with our assumption. Hence $|U|\leq 2n_k{2\lchuto-2\choose \lchuto}$.

        Note that $|A|\leq (k+1)\lchuto-3<2(k-1)\lchuto$ when $k\geq 3$. By Lemma \ref{lemma: edges inside BCD} and the fact $e(G[D])\leq \frac{l-2}{2}|D|$, we have
        \begin{equation*}
            \begin{aligned}
               & e(G)= e(G[A])+e(A,U)+e(A,B\backslash U)+(e(G[C])+e(C,V(G)\setminus C))\\
                &+(e(G[D])+e(D,V(G)\setminus D))+e(G[B])
                \\\leq & {|A|\choose 2}+\left((k+1)\lchuto-3\right)|U|+\left(k\lchuto-2\right)(n-|A|-|C|-|D|-|U|)\\
                &+\left(1+(k-1)\lchuto-1\right)|C|+\left(1+\frac{\ell-2}{2}\right)|D|+1   \\
                \leq & {2(k-1)\lchuto\choose 2}+\left((k+1)\lchuto-3\right)2n_k{2\lchuto-2\choose \lchuto}\\
                &+\left(k\lchuto-2\right)\left(n-2n_k{2\lchuto-2\choose \lchuto}\right)+2
                     \end{aligned}
        \end{equation*}
        \begin{equation*}
            \begin{aligned}
                =&{2(k-1)\lchuto\choose 2}+2\left(\lchuto-1\right)n_k{2\lchuto-2\choose \lchuto}+\left(k\lchuto-2\right)n+2\\
                < &\left(n-k\lchuto+1\right)\left(k\lchuto-1\right),
            \end{aligned}
        \end{equation*}
        a contradiction with (\ref{eq-0}).
        The last inequality holds because
        \begin{equation*}
            \begin{aligned}
                2\left(\lchuto-1\right)n_k{2\lchuto-2\choose \lchuto}=&2\left(\lchuto-1\right)\left(\frac{n/3-{\ell\choose 2}}{(\lchutoo+1){\ell \choose \lchuto}}+k\ell\right){2\lchuto-2\choose \lchuto}\\
                \leq & \frac{2}{3}n+2k\ell(\lchuto-1){2\lchuto-2\choose \lchuto}<\frac{3}{4}n,
            \end{aligned}
        \end{equation*}
        when $n\geq 4k^2\ell^2(k\ell+r)(\lchuto+1){\ell\choose \lchuto}$ and $k\geq 3$.

       Now we assume for any $E\in E(\mathscr{H}[A\setminus S^x])$, we have $x\notin E$.
        Then we will replace $A_1$ with $A_1':=V(\mathscr{H}[A\setminus(S_0\cup S_1)])$. Thus $G[V(G)\setminus (S_0\cup S_1)]$ is $(k-1)P_\ell$-free and  $A\setminus(S_0\cup S_1)=A_1'$.  Then by a similar analysis as above, we have $|A_1'|\leq (k-1)\lchuto-1$ and $|A|\leq (k+1)\lchuto-2.$ We can also prove $e(G)< (k\lchuto-1)n$ similarly.\q
    \vskip.2cm
    \noindent{\bf Claim 4.} There exists  $E_0\in E(\mathscr{H})$ such that $E_0\cap S^x\neq \emptyset$ and $E_0\cap(\cup_{i=2}^{k-1}S_i)\neq \emptyset$.

\noindent{\bf Proof of Claim 4.} By Lemma  \ref{lem: delete t vertices always have k-1 edges},  $\nu(\mathscr{H}[A\setminus\{x\}])\ge k-1$. Since $\nu(\mathscr{H})=k-1$,  there exists $E_0\in E(\mathscr{H}[A\setminus\{x\})])\subseteq E(\mathscr{H})$ such that $E_0\cap S^x\neq \emptyset$. If $E_0\cap(\cup_{i=2}^{k-1}S_i)\neq \emptyset$, then we are done.

If $E_0\cap(\cup_{i=2}^{k-1}S_i)= \emptyset$, by Claim 2, we have $E_0\subseteq S^x$. By Claim 3,    there exists $S'\in E(\mathscr{H})$ with $S'\cap S^x\neq \emptyset$ and $S'\not\subseteq S_0\cup S_1$. Together with Claim 2, we have $S'\cap(\cup_{i=2}^{k-1}S_i)\neq \emptyset$ and we are done.\q

 \vskip.2cm

 By Claim 4, let $E_0\in E(\mathscr{H})$ such that $E_0\cap S^x\neq \emptyset$ and $E_0\cap(\cup_{i=2}^{k-1}S_i)\neq \emptyset$. Assume $E_0\cap (S_1\setminus\{x\})\neq \emptyset$ and $E_0\cap S_2\neq \emptyset$.
 Let  $z'\in E_0\cap S_2$  and $x_{z'}\in (E_0\cap S_1)\setminus\{x\}$.

 \vskip.2cm
  \noindent{\bf Claim 5.} For any $y\in \overline{A}$, $|S_y\cap S_2|=1$. Particularly, $z'\in S_y\cap S_2$.

  \noindent{\bf Proof of Claim 5.} By Claim 2, $S_y\cap S_2\not=\emptyset$. Suppose $y_1,y_2\in S_y\cap S_2$. Then we can assume that $z'\not=y_1$.  Then we can find two disjoint copies of $P_{\lchuto}$ in $G_{\mathscr{H}}[S_0\cup S_1\cup\{z'\}]$ and a copy of $P_{\lchuto}$ in $G_{\mathscr{H}}[ (S_2\setminus\{z'\})\cup\{y\}]$. Note that there are $k-3$ disjoint copies of $P_{\lchuto}$ in $G_{\mathscr{H}}[\cup_{i=3}^{k-1}S_i]$. So there are $k$ disjoint copies of $P_{\lchuto}$ in $G_{\mathscr{H}}$, a contradiction.\q

 Now we are going to complete the proof of Lemma \ref{lem: the size of A}. Let $y\in \overline{A}$.
Since $|S_y\setminus\{y\}|\leq \lchuto-1$, by Lemma \ref{lem: delete t vertices always have k-1 edges}, there exist $k-1$ disjoint hyperedges, denoted as $H_1,H_2,\dots, H_{k-1}$, in $\mathscr{H}[(A\setminus S_y)\cup \{y\}]$. Since $\nu(\mathscr{H})=k-1$, $\{H_1,H_2,\dots, H_{k-1}\}$ is a maximum matching in $\mathscr{H}$ which implies $y\in \cup_{i=1}^{k-1}H_i$, say $y\in H_1$; otherwise, together with $S_y$, we have $\nu(\mathscr{H})=k$, a contradiction. By Claim 2, $H_1\cap S_2\not=\emptyset$. By Claim 5, $z'\in S_y\setminus\{y\}$. Since $\{H_1,H_2,\dots, H_{k-1}\}$ is a maximum matching in $\mathscr{H}[(A\setminus S_y)\cup \{y\}]$, $z'\notin H_1$. So $z'\notin H_1\cap S_2$, a contradiction with Claim 5.
\end{proof}

\vskip.2cm
Recall that $G$ is the maximum $\{kP_\ell,F\}$-free graph, where $v(F)=r$. We choose $G$ such that $G$ has the minimum number of edges in $G[V(G)\setminus A]$ among all extremal graphs.
\begin{lemma}\label{lastlem}
(1) $|A|= k\lchuto -1$;
(2) $|B'|\ge k\ell+r$, where $B'=\{u\in B: d_A(u)=k\lchuto -1\}$; (3) $|C|=|D|=0$; (4) $N(u)\cap A=A$  for any $u\in B$ and $B=B'$.
\end{lemma}
\begin{proof} (1) Let $X\subseteq A$ with $|X|= \lchuto-1$. By Lemma
\ref{lem: delete t vertices always have k-1 edges}, $\nu(\mathscr{H}[A\setminus X])\ge k-1$ which implies $|A\setminus X|\ge (k-1)\lchuto$. By Lemma \ref{lem: the size of A}, $|A|= k\lchuto -1$.

(2) Suppose $|B'|\le k\ell+r-1$. By Lemma \ref{lemma: edges inside BCD} and the fact $e(G[D])\leq \frac{l-2}{2}|D|$, we have
\begin{equation*}
    \begin{aligned}
        e(G)=& e(G[A])+e(A,B)+e(G[C])+e(C,V(G)\setminus C)+e(G[D])+e(D,V(G)\setminus D)+e(G[B])\\
        \leq & e(G[A])+\left(k\lchuto-2\right)\left(n-\left(k\lchuto-1\right)-|C|-|D|-|B'|\right)\\
        +&\left(k\lchuto-1\right)|B'|+ \left(1+(k-1)\lchuto -1\right)|C|+\left(1+\frac{\ell-2}{2}\right)|D|+1\\
        \leq &  {k\lchuto-1\choose 2}+\left(k\lchuto-2\right)\left(n-k\lchuto+1\right)\\
        -& \left(\lchuto-1\right)|C|-\left(k\lchuto-2-\frac{\ell-2}{2}\right)|D|+\left(k\lchuto-1\right)(k\ell+r-1)+1\\
        <&\left(k\lchuto-1\right)\left(n-k\lchuto+1\right),
    \end{aligned}
\end{equation*}
a contradiction with (\ref{eq-0}). 

(3) Suppose $C\not=\emptyset$ or $D\not=\emptyset$. Let $G'$ be the graph obtained from $G$ by deleting all vertices in $C\cup D$,  adding $|C|+|D|$ new isolate vertices (denote the set by $Y$),  and then connecting all these new vertices  with all vertices in $A$. By Lemma \ref{lemma: edges inside BCD} and the definitions of $C$ and $D$, $e(G')>e(G)$. If there exists a copy of $F$ or a copy of $kP_\ell$ in $G'$, then it will contain at least one and at most $k\ell+r$ vertices in $Y$. Since $N(u)\cap A=A$ for any $u\in B'$ and $|B'|\ge kP_\ell$ by (2), there is a copy of $F$ or $kP_\ell$ in $G$. Thus $G'$ is $\{k P_\ell,F\}$-free. But $e(G')>e(G)$, a contradiction.

(4) By Lemma \ref{lemma: edges inside BCD}, $e(G[B])\le 1$. Let $u\in B$. If $d_{G[B]}(u)=0$, by the same argument as that proof of (3), (4) holds. Assume $u_1u_2\in E(G[B])$ when $e(G[B])= 1$. Suppose $|N(u_i)\cap A|<|A|$, say $|N(u_1)\cap A|<|A|$. Let $G'$ be the graph obtained from $G$ by deleting $u_1u_2$ and connecting $u_1,u_2$ with all the vertices in $A$. By the same argument as that proof of (3), $G'$ is $\{k P_\ell,F\}$-free. But $e(G')\geq e(G)$ and $e(G'[V(G)\setminus A])<e(G[V(G)\setminus A])$, a contradiction.
\end{proof}

\vskip.2cm
Now we can prove the main result in the section.
\vskip.2cm
\begin{provekplF} By Lemma \ref{lastlem}, we have
$$e(G)=e(G[A])+e(G[B])+\left(k\lchuto-1\right)\left(n-k\lchuto+1\right).$$Recall $\mathscr{G}_1(F)=\{F[V(F)\backslash S]\mid S\subseteq V(F), e(F[S])=0\},$ and
$\mathscr{G}_2(F)=\{F[V(F)\backslash S]\mid S\subseteq V(F),e(F[S])\leq 1\}.$

Let  $\ell$ be an odd number. By Lemma \ref{lemma: edges inside BCD}, $e(G[B])\le 1$. Since $G$ is $\{k P_\ell,F\}$-free, by Lemma \ref{lastlem} (2) and (4), $e(G[A])\leq ex(|A|,\mathscr{G}_1(F))$
if $e(G[B])=0$ and $e(G[A])\leq ex(|A|,\mathscr{G}_2(F))$ if $e(G[B])=1$.

When $\ell$ is even, we have $e(G[B])= 0$ by Lemma \ref{lemma: edges inside BCD}. Then we also have
$e(G[A])\leq ex(|A|,\mathscr{G}_1(F)) $.
Thus we have
$$e(G)\leq \left(n-k\lchuto+1\right)\left(k\lchuto-1\right)+b_\ell,$$ where  $b_\ell=\max\{1+ex(k\lchuto-1,\mathscr{G}_2(F)),ex(k\lchuto-1,\mathscr{G}_1(F))\} $ when $\ell$ is odd and $b_\ell= ex(k\lchuto-1,\mathscr{G}_1(F))$ when $\ell$ is even.
\end{provekplF}

\vskip.2cm
\begin{provekplF=1}
    Similarly, we have
    $$e(G)\leq e(G[A])+e(G[B])+\left(k\lchuto-1\right)\left(n-k\lchuto+1\right).$$
    Recall that when $\beta_1(F)=1$, $F$ is a subgrapgh of $B_r$ where $r=v(F)$ and $B_r$ is the graph constructed by $r$ triangles sharing one edge.
    Suppose $e(G[A])\geq 1$, say $u_1u_2\in E(G[A])$.
    By Lemma \ref{lastlem}, $u_1u_2$ has at least $k\ell+r$ common neighbourhoods. Thus there is a copy of $B_r$ in $G$ which implies there is a copy of $F$  in $G$, a contradiction.

    We claim that if $(F,k,\ell)$ has property $\mathscr{P}$, then $e(G[B])=0$. By Lemma \ref{lemma: edges inside BCD}, we just need to consider that
     $\ell$ is odd. Suppose $e(G[B])\geq 1$ and $w_1w_2\in E(G[B])$.
    If $r\leq k\lchuto+1$, by Lemma \ref{lastlem}, we can similarly find a copy of $F$ in $G$, a contradiction. Since $(F,k,\ell)$ has property $\mathscr{P}$, we have
     $N(F,K_3)=1$ and $f(F)\leq k\lchuto$. Let $v\in A$. By Lemma \ref{lastlem}, $vw_1w_2v$ is a copy of $K_3$.  Since $f(F)\leq k\lchuto$, $d_G(v)\geq k\ell+r$ and $d_G(w_i)\geq |A|+1=k\lchuto$ for $i=1,2$, we can find a copy of $F$ in $G$, a contradiction. Hence, $e(G[B])=0$ if $(F,k,\ell)$ has property $\mathscr{P}$. Thus, if  $(F,k,\ell)$ has property $\mathscr{P}$, we have
    $$e(G)\leq\left(k\lchuto-1\right)\left(n-k\lchuto+1\right).$$

    \noindent If $(F,k,\ell)$ does not have property $\mathscr{P}$, we have
    $$e(G)\leq \left(k\lchuto-1\right)\left(n-k\lchuto+1\right)+1.$$
\end{provekplF=1}


\section{Proof of Theorem   \ref{p3}}

Let $G$ be a $\{k P_3,F\}$-free graph  of order $n$ with $ex(n,\{k P_3,F\})$ edges, where $n \geq 9(k^2+k+1)+2v(F)$. According to Theorem  \ref{thm: ex(n,kP3)}, $ex(n,k P_3) = (k-1)(n-k+1) + \left\lfloor \frac{n-k+1}{2} \right\rfloor + {k-1\choose 2}$, where $k\ge 2$.  Notice that $K_{k-1,n-k+1}$ is $\{k P_3,F\}$-free. We have
\[
(k-1)(n-k+1) \leq ex(n,\{k P_3,F\}) \leq (k-1)(n-k+1) + \left\lfloor \frac{n-k+1}{2} \right\rfloor + \binom{k-1}{2}.
\]
Since
\begin{align*}
    ex(n,\{k P_3,F\}) &\geq (k-1)(n-k+1)  \\&> (k-2)(n-k+2) + \left\lfloor \frac{n-k+2}{2} \right\rfloor + \binom{k-2}{2} \\&\geq ex(n,\{(k-1) P_3,F\}),
\end{align*}there are $k-1$ disjoint copies of $P_3$, say $P_3^1,\ldots,P_3^{k-1}$, in $G$.  Then $G[V(G)\setminus V(P_3^i)]$ is $(k-1) P_3$-free for $1\le i\le k-1$. Hence for any $1\le i\le k-1$, we have
\begin{align*}
    &e(G)-e(G[V(G)\setminus V(P_3^i)])\geq ex(n,\{k P_3,F\})-ex(n-3,\{(k-1) P_3,F\})\\
     &\geq (k-1)(n-k+1)-\left((k-2)(n-k-1)+\lf\frac{n-k-1}{2}\rf+\binom{k-2}{2}\right)
     \geq\frac{n}{3}.
\end{align*}
Thus there is  $y_i \in V(P_3^i)$ such that $d(y_i)\ge \frac{n}{9}\ge k^2+k+1+\frac{2v(F)}{9}$ for any $1\le i\le k-1$. Let $A=\{y_1,\ldots,y_{k-1}\}$. We have the following lemmas.
\begin{lemma}\label{lemmaA}
     $\Delta(G[V(G) \setminus A]) \leq 1$.
\end{lemma}
\begin{proof} Suppose $\Delta(G[V(G) \setminus A]) \geq 2$. Then there is a copy of $P_3$ in $G[V(G) \setminus A]$, say $P'$. Since $d(y_i)\ge \frac{n}{9}> k^2+k+1$ for any $1\le i\le k-1$, there are  $(k-1)$ disjoint copies of $P_3$ in $G[V(G)\setminus V(P')]$, a contradiction with  $G$ being $k P_3$-free.
\end{proof}

\begin{lemma}\label{infinity}
    If $\sigma(F) = \infty$, then $N(u)\cap A=A$ for any $u\in V(G)\setminus A$  and $e(G[V(G) \setminus A]) = \left\lfloor \frac{n-k+1}{2} \right\rfloor$.
\end{lemma}
\begin{proof}
    Since $\sigma(F) = \infty$,  $I_{k-1} + M_{n-k+1}^{\left\lfloor \frac{n-k+1}{2} \right\rfloor}$ is $\{k P_3,F\}$-free. We immediately get
    \begin{equation}\label{eq-1}
    e(G)=ex(n,\{k P_3,F\}) \geq (k-1)(n-k+1) + \left\lfloor \frac{n-k+1}{2} \right\rfloor.
    \end{equation}

   Let $X=\{u\in V(G)\setminus A:N(u)\cap A=A\}$. By Lemma \ref{lemmaA}, $e(G[V(G) \setminus A]) \le  \left\lfloor \frac{n-k+1}{2} \right\rfloor$. If $|X|\le n-k-\binom{k-1}{2}$, then 
    \begin{align*}
        e(G) &= e(G[A]) + e(G[V(G) \setminus A]) + e(A, V(G) \setminus A) \\
        &\leq \binom{k-1}{2} + \left\lfloor \frac{n-k+1}{2} \right\rfloor + (k-1)(n-k+1) - \binom{k-1}{2} - 1\\
        &= \left\lfloor \frac{n-k+1}{2} \right\rfloor + (k-1)(n-k+1) -1,
            \end{align*}
     a contradiction with (\ref{eq-1}). Hence $|X|\ge n-k+1-\binom{k-1}{2}$. Note that $e(G[X])\geq \frac{1}{2}(|X|-|V(G)\setminus(A\cup X)|)\geq \frac{1}{2}(n-k+1-2\binom{k-1}{2}).$ Since $n\geq 9(k^2+k+1)+2v(F)$, $e(G[X])\geq v(F)$.

     Suppose $V(G) \setminus A\not=X$. Let $u\in V(G) \setminus (A\cup X)$. Let $G'$ be the graph with $V(G')=V(G)$ and $E(G')=E(G)\cup \{ux|ux\notin E(G) \mbox{~and~}x\in A\}$. Then $G'$ is $kP_3$-free. If there is a copy of $F$ in $G'$, say $F'$, then $u\in V(F')$. Since $e(G[X])\ge v(F)$ and $\Delta(G[V(G) \setminus A]) \leq 1$ (Lemma \ref{lemmaA}), we can choose $v\in X\setminus N_X(u)$ with $N_X(v)\not=\emptyset$ such that $G[(V(F')\setminus\{u\})\cup\{v\}]$ is a copy of $F$, a contradiction. Thus $G'$ is $\{k P_3,F\}$-free but $e(G')>e(G)$, a contradiction. Hence $V(G) \setminus A=X$.





    Suppose $e(G[V(G) \setminus A]) <  \left\lfloor \frac{n-k+1}{2} \right\rfloor$. Then there are two vertices $u, v \in V(G) \setminus A$ such that $N_G(u) \setminus A = N_G(v) \setminus A = \emptyset$. Let $G'=G+uv$. By the same argument as above, we have $G'$ is $\{k P_3,F\}$-free but $e(G')>e(G)$, a contradiction.
\end{proof}

\begin{lemma}\label{ninfinity}
    Let $G$ be the maximum $\{k P_3,F\}$-free graph with minimum $e(G[V(G)\setminus A])$. If $\sigma(F) < \infty$, then  $N(u)\cap A=A$ for any $u\in V(G)\setminus A$  and $e(G[V(G) \setminus A]) \leq \sigma(F)$.
\end{lemma}
\begin{proof}
    Since $I_{k-1} + M_{n-k+1}^{\sigma(F)}$ is $\{k P_3,F\}$-free, we immediately get
    \begin{equation}\label{eq-2}
    ex(n,\{kP_3,F\}) \geq (k-1)(n-k+1) + \sigma(F).
\end{equation}
    From the definition of $\sigma(F)$ and Lemma \ref{lemmaA},
     $e(G[V(G) \setminus A])\le \sigma(F)$. Otherwise, there would either be a copy of $F$ in $G$ or a copy of $P_3$ in $G[V(G) \setminus A]$, a contradiction.  Let $X=\{u\in V(G)\setminus A:N(u)\cap A=A\}$. By Lemma \ref{lemmaA}, $e(G[V(G) \setminus A]) \le  \left\lfloor \frac{n-k+1}{2} \right\rfloor$. If $|X|\le n-k-\binom{k-1}{2}$, then
    \begin{align*}
        e(G) &= e(G[A]) + e(G[V(G) \setminus A]) + e(A, V(G) \setminus A) \\
        &\leq \binom{k-1}{2} + \sigma(F) + (k-1)(n-k+1) - \binom{k-1}{2} - 1 \\
        &= (k-1)(n-k+1) + \sigma(F)-1,
    \end{align*}
     a contradiction with (\ref{eq-2}). Hence $|X|\ge n-k+1-\binom{k-1}{2}>v(F)$ by $n \geq 9(k^2+k+1)+2v(F)$.

     Suppose $V(G) \setminus A\not=X$. Let $u\in V(G) \setminus (A\cup X)$. Assume $e=uv$ if $d_{G[V(G) \setminus A]}(u)=1$; otherwise let $e=\emptyset$. Let $G'$ be the graph with $V(G')=V(G)$ and $E(G')=(E(G)\setminus\{e\})\cup \{ux|ux\notin E(G) \mbox{~and~}x\in A\}$. Then $G'$ is $kP_3$-free. If there is a copy of $F$ in $G'$, say $F'$, then $u\in V(F')$. Since $|X|\ge v(F)$, we can choose $v\in X$  such that $G[(V(F')\setminus\{u\})\cup\{v\}]$ is a copy of $F$, a contradiction. Thus $G'$ is $\{k P_3,F\}$-free but $e(G')>e(G)$,  or $e(G') = e(G)$ and $e(G'[V(G) \setminus A]) < e(G[V(G) \setminus A])$, a contradiction.
\end{proof}

\vskip.2cm
Now we complete the proof of Theorem  \ref{p3}.

\noindent\textit{Proof of Theorem  \ref{p3}} We first consider the case $\sigma(F) = \infty$. Let $G$ be a $\{k P_3,F\}$-free graph  of order $n$ with $ex(n,\{k P_3,F\})$ edges, where $n \geq 9(k^2+k+1)+2v(F)$. From Lemma \ref{infinity},  we can only add edges between the vertices in $A$. To avoid creating any copy of $F$, at most $ex(k-1, \mathscr{H}(F))$ edges can be added. Hence,
\[
ex(n,\{k P_3,F\}) = (k-1)(n-k+1) + \left\lfloor \frac{n-k+1}{2} \right\rfloor + ex(k-1, \mathscr{H}(F)).
\]

Now we consider the case $\sigma(F) < \infty$. Let $G$ be a $\{k P_3,F\}$-free graph  of order $n$ with $ex(n,\{k P_3,F\})$ edges and minimum $e(G[V(G) \setminus A])$, where $n \geq 9(k^2+k+1)+2v(F)$. Let $e(V(G) \setminus A) = i$. Then $i \leq \sigma(F)$ by Lemma  \ref{ninfinity}. From Lemma  \ref{ninfinity}, $G$ can be obtained by adding edges between the vertices in $A$. To avoid creating any copy of $F$, at most $ex(k-1, \mathscr{H}_i(F))$ edges can be added. Hence, by considering all $i$, we get
\[
ex(n,\{k P_3,F\}) = \max_{0 \leq i \leq \sigma(F)} \left\{(k-1)(n-k+1) + i + ex(k-1, \mathscr{H}_i(F))\right\}.
\]
\hfill$\square$

\section{Proofs of Theorems \ref{thm: linear forest and beta 2} and  \ref{thm: linear forest and beta 1}}
Let $H=(\cup_{i\in I_1}P_{2\ell_i})\cup (\cup_{j\in I_2}P_{2\ell_j+1})$, where $I_1$ (resp. $I_2$) is the collection of indices of even paths (resp. odd paths) and $|I_1|+|I_2|=k\ge 2$.
Let $\ell=\ell_1+\ell_2+\dots+\ell_k$.


\subsection{The low bounds}
In this subsection, we give the lower bound of $ex(n,\{H,F\})$.

We first consider the low bound of Theorem \ref{thm: linear forest and beta 2}. Then $\beta_1(F)\geq 2$.
Let $H_1'(n,k,\ell,F)$ be the graph obtained by embedding a $\mathscr{G}_1(F)$-free graph of size $ex(\ell-1,\mathscr{G}_1(F))$ to the smaller part of $K_{\ell-1,n-\ell+1}$. Let $H_2'(n,k,\ell,F)$ be the graph obtained by embedding a $\mathscr{G}_2(F)$-free graph of size $ex(\ell-1,\mathscr{G}_2(F))$  to the smaller part of $K_{\ell-1,n-\ell+1}$ and adding one edge in the larger part. It is easy to check $H_1'(n,k,\ell,F)$ is $\{H,F\}$-free regardless of whether $I_1$ is an empty set or not, and $H_2'(n,k,\ell,F)$ is $\{H,F\}$-free if $I_1=\emptyset$. Hence we have $$ex(n,\{F,H\})\ge \left(n-\ell+1\right)\left(\ell-1\right)+b_\ell',$$where  $b_\ell'=\max\{1+ex(\ell-1,\mathscr{G}_2(F)),ex(\ell-1,\mathscr{G}_1(F))\} $ when $I_1=\emptyset$ and $b_\ell'= ex(\ell-1,\mathscr{G}_1(F))$ otherwise.

Now we consider the case $\beta_1(F)=1$, that is, the low bound of Theorem \ref{thm: linear forest and beta 1}.
Note that $K_{\ell-1,n-\ell+1}$ is $\{H,F\}$-free.
Then  $K_{\ell-1,n-\ell+1}$ provides the lower bound of $ex(n,\{H,F\})$ when $(F,H)$ has condition $\mathscr{R}$. So  $ex(n,\{H,F\})\ge (\ell-1)(n-\ell+1)$ if $\beta_1(F)=1$ and $(F,H)$ has condition $\mathscr{R}$.

Let $\beta_1(F)=1$ and $(F,H)$ do not have condition $\mathscr{R}$. Then $I_1=\emptyset$, $r=v(F)>\ell+1$, and $N(F,K_3)\ge 2$ or $f(F)> \ell$. Adding one edge to the larger part of $K_{\ell-1,n-\ell+1}$, and we denote the obtained graph by $G_0'$. Then $e(G_0')=(\ell-1)(n-\ell+1)+1$. It is easy to check that  $G_0'$ is $H$-free by $I_1=\emptyset$.
We will prove it is $F$-free. Let $\{v_1v_2\}$ be the edge control set of $F$.

If $N(F,K_3)\geq 2$, then $v_1v_2$ is contained in at least two copies of $K_3$. If $G_0'$ has a copy of $F$, then  $v_1v_2$ must be the unique edge  in the larger part. Since $|N_{G_0'}(v_1)\cup N_{G_0'}(v_2)|=\ell-1+2<r$, we have a contradiction with $|N_F(v_1)\cup N_F(v_2)|=r$.
If $N(F,K_3)=1$, then  $f(F)\geq \ell+1$. In this case, $v_1v_2$ is contained in the $K_3$. If $v_1v_2$ is the unique edge in the larger part, then we can similarly have a contradiction. If  $v_1$ is in the smaller part of $G_0'$ and $v_2$ belongs to the larger part of $G_0'$, then $\min\{d_{G_0'}(v_1),d_{G_0'}(v_2)\}=d_{G_0'}(v_2)\leq \ell$, but $f(F)=\min\{d_F(v_1),d_F(v_2)\}\geq \ell+1$, a contradiction. Thus $G_0'$ is $F$-free. So  $ex(n,\{H,F\})\ge (\ell-1)(n-\ell+1)+1$ if $\beta_1(F)=1$ and $(F,H)$ has no condition $\mathscr{R}$.
\subsection{The upper bound}

Let  $H=\cup_{i=1}^k H_i$, where $H_i$ is a path, $k\ge 2$  and $v(H_i)\ge 3$ for $1\le i\le k$. Let $\ell_i=\lfloor\frac{v(H_i)}{2}\rfloor$ for $1\le i\le k$ and $\ell=\ell_1+\ell_2+\dots+\ell_k$. Assume that $\ell_1\leq \ell_2\leq \dots\leq \ell_k$. Let  $\ell'=\ell-\ell_1$, $h=v(H)=\sum_{i\in I_1}2\ell_i+\sum_{j\in I_2}(2\ell_j+1)$, $H'=H_2\cup \dots \cup H_k$ and $h'=v(H')=h-v(H_1)$, where $I_1$ (resp. $I_2$) is the collection of indices of even paths (resp. odd paths). Since $v(H_i)\ge 3$, we have $\ell_i\geq 2$ if $i\in I_1$ and $\ell_j\geq 1$ if $j\in I_2$.
Since $k\geq 2$ and $H\neq kP_3$, we have $\ell_k\geq 2$. By Theorem \ref{thm: ex(n,kpl,F)}, we will assume $\ell_1<\ell_k$ when $I_1=\emptyset$.

Let $G$ be a $\{H,F\}$-free graph with $n$ vertices and $ex(n,\{H,F\})$ edges. Since $K_{\ell -1,n-\ell +1}$ is an $\{H,F\}$-free graph,  $e(G)\geq (\ell -1)(n-\ell +1)$. In this subsection, we always suppose $n$ is large enough. Then by Theorem \ref{thm: old linear forest}, $e(G)\geq (\ell-1)(n-\ell+1)> ex(n,H')$.
\begin{lemma}\label{lem: bounded edges of E(H',V-H')}
    For any copy of $H'$ in $G$, we have $e(V(H'),V(G)\setminus V(H'))\geq \ell'n-6\ell^2$. Moreover, there exists $A\subseteq V(H')$ such that $|A|=\ell'$ and $|N(A)|\ge n':=\frac{n-6\ell^2}{2(\ell'+1){h'\choose \ell'}}$.
\end{lemma}
\begin{proof}
    Since  $e(G)> ex(n,H')$, there is a copy of $H'$, also denoted by $H'$,  in $G$. Then $G[V(G)\setminus V(H')]$ is $H_1$-free by $G$ being $H$-free. Note that   $h'\le 3\ell'$. Then 
    $$e(V(H'),V(G)\setminus V(H'))\geq e(G)-\binom{3\ell'}{2}-ex(n,H_1)\geq (\ell-1)(n-\ell+1)-\binom{3\ell'}{2}-\ell_1n\geq \ell'n-6\ell^2.$$
    Let $X:=\{v\in V(G)\setminus V(H'):~|N(v)\cap V(H')|\geq \ell'\}$. Then $e(V(H'),V(G)\setminus V(H'))\le3\ell'|X|+(\ell'-1)n$. By
    $3\ell'|X|+(\ell'-1)(n-|X|)\geq  \ell'n-6\ell^2,$
     we have $|X|\geq \frac{n-6\ell^2}{2\ell'+1}$. As there are only $\binom{h'}{\ell'}$ sets of $\ell'$ vertices in $H'$, there is $A\subseteq V(H')$ such that $|A|=\ell'$ and $|N(A)|\ge \frac{n-6\ell^2}{(2\ell'+1){h'\choose \ell'}}>\frac{n-6\ell^2}{2(\ell'+1){h'\choose \ell'}}$.
\end{proof}
For a rough bound, we may assume $e(G)\leq 3\ell n$ by Theorem \ref{thm: old linear forest}. Let
$$c=\frac{6\ell n}{n'/2}=\frac{24\ell n}{(n-6\ell^2)/((\ell'+1){h'\choose \ell'})}\leq 25\ell(\ell'+1){h'\choose \ell'},$$
when $n$ is large enough. Then $c$ is bounded by a constant related to $\ell'$ and $h'$. Set
$D_{\le c}=\{u\in V(G):d(u)\le c\}$ and $D_{>c}=\{u\in V(G):d(u)> c\}$. Then $|D_{>c}|\le n'/2$ by $e(G)\leq 3\ell n$.
Since  $e(G)\geq (\ell-1)(n-\ell+1)>ex(n,H')$, by Theorem \ref{thm: old linear forest}, there is a copy of $H'$ in $G$. By   Lemma \ref{lem: bounded edges of E(H',V-H')}, there is $A\subseteq V(H')$ with $|A|=\ell'$ such that $|N(A)|\ge n'$.  Since $|D_{>c}|\le n'/2$ and $|N(A)|\ge n'$, there is $A'\subseteq N(A)\cap D_{\le c}$ with $|A'|=h'-\ell'$  such that there is a copy of $H'$ in $G[A\cup A']$.
We call a subgraph $G[A\cup A']$ {\em pseudo-bipartite graph} if $|A|=\ell'$, $|A'|=h'-\ell'$ and $A'\subseteq N(A)\cap D_{\le c}$.
\begin{lemma}\label{lem: if a copy H' contains many small degree vertices, then l' vertices have large degree}
If $G[A\cup A']$ is a pseudo-bipartite graph, then $d(v)\ge n-c'$ for all $v\in A$, where $c'=6\ell^2+\ell'c$.
\end{lemma}
\begin{proof} Let $H_A$ be a copy of $H'$ in $G[A\cup A']$.
   For any  $v\in A$, we have
    \begin{equation*}
        \begin{aligned}
            d(v)= & e(V(H_A),V(G)\setminus V(H_A))-e(A',V(G)\setminus V(H_A))-e(A\setminus\{v\},V(G))\\
            \geq &\ell'n-6\ell^2-(h'-\ell')c-(\ell'-1)n\\
            \geq & n-6\ell^2-\ell'c=n-c'.
        \end{aligned}
    \end{equation*}
    The second inequality follows from Lemma \ref{lem: bounded edges of E(H',V-H')}.
\end{proof}
By Lemma \ref{lem: if a copy H' contains many small degree vertices, then l' vertices have large degree}, $|N(A)|\ge n-\ell'c'$. 
 Let $B=\{v\notin A:d(v)\ge \frac{n-3\ell^2}{2(\ell_k+1)}\}$. Assume $B\not=\emptyset$. For any $y\in B$, we have
$$|N(y)\cap N(A)\cap D_{\le c}|\ge \frac{n-3\ell^2}{2(\ell_k+1)}-\frac{n'}{2}\geq \frac{n-3\ell^2}{8(\ell_k+1)}>\ell'c'+h'.$$Let $x\in A$ and $y\in B$. Denote $A_y=(A\setminus\{x\})\cup \{y\}$. Since $|N(y)\cap N(A)\cap D_{\le c}|>\ell'c'+h'$, there is $A_y'\subseteq  N(y)\cap N(A)\cap D_{\le c}$ with $|A_y'|=h'-\ell'$ such that $G[A_y\cup A_y']$ is a pseudo-bipartite graph. By Lemma \ref{lem: if a copy H' contains many small degree vertices, then l' vertices have large degree}, $d(y)\ge n-c'$.

We claim that $|B|\leq \ell_1-1$.
 Suppose $|B|=s\geq \ell_1$. Since $d(y)\ge n-c'$ for any $y\in B$, $|N(B)|\ge n-sc'$. So there is $B'\subseteq N(B)\setminus (A'\cup A)$ such that $|B'|= \ell_1+1$. Thus there is a copy of $H$ in $G[A\cup B\cup A'\cup B']$, a contradiction.

Now we have  that $ \ell_k\le |A\cup B|\leq \ell-1$ and for any  $y\in A\cup B$, $d(y)\ge n-c'$.

\begin{lemma}\label{lemma: no H_1 contains l1-1 vertices in AUB}For $1\le i\le k$,
    $|V(H_i)\cap (A\cup B)|\ge \ell_i$.
\end{lemma}
\begin{proof}
   Let $1\le i\le k$ and $H^{(i)}=H[\cup_{j=1}^k V(H_j)\setminus V(H_i)]$. By Lemma \ref{thm: old linear forest}, we have 
    \begin{equation*}
        \begin{aligned}
            &e(V(H_i),V(G)\setminus V(H_i))=e(G)-e(G[V(H_i)])-ex(n-v(H_i),H^{(i)})\\
            \geq & (\ell-1)(n-\ell+1)-{v(H_i)\choose 2}-(n-(\ell-\ell_i)-v(H_i)+1)(\ell-\ell_i-1)-{\ell-\ell_i-1\choose 2}\\
            \geq & \ell_i n-3\ell^2.
        \end{aligned}
    \end{equation*}
    Suppose there is $1\le i\le k$ such that $|V(H_i)\cap (A\cup B)|\le \ell_i-1$. Then
    \begin{equation*}
        \begin{aligned}
    e(V(H_i),V(G)\setminus V(H_i))\le (v(H_i)-\ell_i+1)\left(\frac{n-3\ell^2}{2(\ell_k+1)}-1\right)+(\ell_i-1)n<\ell_i n-3\ell^2,
    \end{aligned}
    \end{equation*} a contradiction.
\end{proof}

Let
\begin{equation*}
    \begin{aligned}
        D_0=&\{v\in V(G)\setminus (A\cup B): d_{A\cup B}(v)=|A\cup B|\},\\
        D_1=&\left\{v\in V(G)\setminus(A\cup B): d_{A\cup B}(v)\geq \max\left\{\frac{\ell+\ell_1-2}{2},\ell_k\right\}\right\},\\
        D_2=&\left\{v\in V(G\setminus(A\cup B): \ell_k\leq d_{A\cup B}(v)\leq  \frac{\ell+\ell_1-4}{2} \right\},\\
        D_3=&\{v\in V(G)\setminus(A\cup B): d_{A\cup B}(v)\leq \ell_k-1 \}.\\
    \end{aligned}
\end{equation*}
Notice that  $D_2=\emptyset$ if $\ell_k>\frac{\ell+\ell_1-4}{2}$ and $D_0\subseteq D_1$. Since there is a copy of $H'$ in $G[A\cup N(A)]$ and $|N(A)|\ge n'$,
$G[D_3]$ is $H_1$-free by $G$ being $H$-free. By Theorem \ref{G}, $e(G[D_3])\leq \frac{v(H_1)-2}{2}|D_3|$.

\begin{lemma}\label{lem: even, no D E}
    If $I_1\neq \emptyset$, then $e(G[D_1\cup D_2])=0$ and $e(D_1\cup D_2,D_3)=0$.
\end{lemma}
\begin{proof} Let $i\in I_1$.
    Suppose there is $uv\in E(G)$ such that $v\in D_1\cup D_2$. Then $d_{A\cup B}(v)\geq \ell_k\geq \ell_i$. Since $d(y)\ge n-c'$ for any $y\in A\cup B$, there are $S\subseteq A\cup B$ and $S'\subseteq V(G)\setminus(A\cup B \cup \{u,v\})$ such that $|S|=\ell_i-1$, $|S'|=\ell_i$ and $S'\subseteq N(S)$. Then $G[\{u,v\}\cup S\cup S']$ contains a copy of $H_i$ with $u$ as an end vertex and $|V(H_i)\cap (A\cup B)|=\ell_i-1$, a contraction with Lemma \ref{lemma: no H_1 contains l1-1 vertices in AUB}.
\end{proof}
\begin{lemma}\label{lem: no P3 containing vertex in D E}
    If $I_1=\emptyset$, then  $d_{D_1\cup D_2}(y)\leq 1$ for every $y\in D_3$ and $\Delta(G[D_1\cup D_2])\leq 1$.
\end{lemma}
\begin{proof} Suppose  there is a vertex $w\in D_1\cup D_2$ such that $w$ is an end vertex of a $P_3$, where $V(P_3)\subseteq D_1\cup D_2\cup D_3$. Then $d_{A\cup B}(w)\geq \ell_k$. By the same argument as that the proof of Lemma \ref{lem: even, no D E}, there is a copy of $H_k$ with only $\ell_k-1$ vertices in $A\cup B$, a contradiction with Lemma \ref{lemma: no H_1 contains l1-1 vertices in AUB}.
\end{proof}

\begin{lemma}\label{lemma: common neighbours}
    We have $|A\cup B|=\ell-1$ whether $I_1\neq \emptyset$ or $I_1= \emptyset$.
\end{lemma}
\begin{proof}
We set $|A\cup B|=t$. Then $t\leq \ell-1$. We first consider the case  $I_1\neq \emptyset$. By Lemma \ref{lem: even, no D E},
     \begin{equation}\label{eq: even}
        \begin{aligned}
            e(G)-e(G[A\cup B])&\leq e(A\cup B,V(G)\setminus(A\cup B))+e(G[D_1\cup D_2])\\
            &+e(G[D_3])+e(D_3,V(G)\setminus D_3)\\
            \leq& t|D_0|+(t-1)(n-t-|D_3|-|D_0|)+\left(\ell_k-1+\frac{2\ell_1-1}{2}\right)|D_3| \\
            =& (t-1)n+|D_0|-t^2-\left(t-\ell_k-\ell_1+\frac{3}{2}\right)|D_3|.
        \end{aligned}
    \end{equation}Since $e(G)\ge (\ell-1)(n-\ell+1)$ and $|D_0|<n$, we have  $t=\ell-1$.

   Assume $I_1=\emptyset$.  By Lemma \ref{lem: no P3 containing vertex in D E},  we have
    \begin{equation}\label{eq: odd}
        \begin{aligned}
            &e(G)-e(G[A\cup B])\leq e(A\cup B,V(G)\setminus(A\cup B))+e(G[D_1\cup D_2])\\
            &+e(G[D_3])+e(D_3,V(G)\setminus D_3)\\
            \leq& t|D_0|+(t-1)(n-t-|D_3|-|D_0|)+\frac{1}{2}(|D_1|+|D_2|)\\
            +&(\ell_k-1)|D_3|+|D_3|+\frac{2\ell_1-1}{2}|D_3|+1\\
            =& |D_0|+(t-1)n-t^2+\frac{1}{2}(|D_1|+|D_2|)-\left(t-\ell_1-\ell_k+\frac{1}{2}\right)|D_3|+1.
        \end{aligned}
    \end{equation}
    Since $e(G)\ge (\ell-1)(n-\ell+1)$, $|D_0|<n$ and $|D_1|+|D_2|+|D_3|<n$, we have  $t=\ell-1$.
    \end{proof}
   \begin{lemma}
       If $k=2$ and $I_1=\emptyset$, then $e(G)-e(G[A\cup B])\leq (\ell-1)(n-\ell+1)+1$.
    \end{lemma}
    \begin{proof} In this case, $\ell=\ell_1+\ell_2$ and $\ell_1<\ell_2$. By Lemma \ref{lemma: common neighbours}, $|A\cup B|=\ell_1+\ell_2-1$.
      Since $I_1=\emptyset$, we have $I_2=\{1,2\}$.
     Set
     \begin{equation*}
         \begin{aligned}
             D_1'=&\left\{v\in V(G)\backslash(A\cup B): d_{A\cup B}(v)\geq \max\left\{\frac{\ell+\ell_1-2}{2},\ell_2\right\}\right\},\\
             D_2'=&\left\{v\in V(G)\backslash(A\cup B): \ell_1\leq d_{A\cup B}(v)<  \max\left\{\frac{\ell+\ell_1-2}{2},\ell_2\right\} \right\},\\
        D_3'=&\{v\in V(G)\backslash(A\cup B): d_{A\cup B}(v)\leq \ell_1-1 \}.\\
         \end{aligned}
     \end{equation*}Note that there is no copy of $H_1$ in $G[D_3']$. By Theorem \ref{G}, $e(G[D_3'])\leq \frac{v(H_1)-2}{2}|D_3'|$.
     We have the following claims.
     \vskip.2cm
    \noindent {\bf Claim 1.}
         There are no independent edges in $G[D_1']$ if $e(G[D_1'])\ge 2$.

   \noindent  {\bf Proof of Claim 1.}
         Suppose there are two independent edges $u_1v_1,u_2v_2$ in $G[D']$. Assume $\ell_1>1$. Then $|N(u_1)\cap N(u_2)|\ge \ell_1-1$. Since $H_{1}\not= P_3$, we can find a copy of $H_{1}$, say $H'_1$ contained $v_1,v_2$ as leaf vertices, such that $|V(H'_1)\cap (A\cup B)|=\ell_1-1$, a contradiction with Lemma  \ref{lemma: no H_1 contains l1-1 vertices in AUB}.

         Assume $\ell_1=1$. Then $\ell=\ell_2+1$ and $d_{A\cup B}(u_i)\geq \ell_2$ for $i=1,2$. So $|(N(u_1)\cap N(u_2))\cap (A\cup B)|\ge \ell_2-1$. Then we can find a copy of $H_2$ with $\ell_2-1$ vertices in $A\cup B$, a contradiction with Lemma \ref{lemma: no H_1 contains l1-1 vertices in AUB}.\q

     \vskip.2cm
     By the same argument  as the proof of Lemma \ref{lem: no P3 containing vertex in D E}, we have the following claim.

     \noindent {\bf Claim 2.} We have
         $d_{D_1'\cup D_2'}(y)\leq 1$ for every $y\in D_3'$ and $d_{G[D_1'\cup D_2']}(v)\leq 1$ for every  $v\in D_1'\cup D_2'$.
    \vskip.2cm
     By Claims 1 and 2, we have
     \begin{equation}\label{eq: k=2}
        \begin{aligned}
            e(G)-e(G[A\cup B])\leq&+ e(A\cup B,V(G)\setminus(A\cup B))+e(G[D_1'\cup D_2'])\\
            &+e(G[D_3'])+e(D_3',V(G)\setminus D_3')\\
            \leq& (\ell-1)(n-\ell-1-|D_2'|-|D_3'|)+\frac{1}{2}|D_2'|+\frac{\ell+\ell_1-4}{2}|D_2'|\\
            +&(\ell_1-1)|D_3'|+|D_3'|+\frac{2\ell_1-1}{2}|D_3'|+1\\
            =& (\ell-1)n-(\ell-1)^2-(\ell-\frac{\ell+\ell_1}{2}+\frac{1}{2})|D_2'|-(\ell-2\ell_1-\frac{1}{2})|D_3'|+1.
        \end{aligned}
    \end{equation}
    Since $\ell_1<\ell_2$, we have $\ell-2\ell_1-\frac{1}{2}\geq 0$ and $\ell-\frac{\ell+\ell_1}{2}+\frac{1}{2}\geq 0$. Thus  we have
    $$e(G)-e(G[A\cup B])\leq (\ell-1)(n-\ell+1)+1.$$
     \end{proof}
    Moreover, we have the following result.
    \begin{lemma}\label{claim: k=2,D=D0}
        If $k=2$ and $I_1=\emptyset$, then we may assume $|D_2'|=|D_3'|=0$ and $D_1'=D_0$.
    \end{lemma}
    \begin{proof} Recall $v(F)=r$.
        Since $d(y)\ge n-c'$ for any  $y\in A\cup B$, we have $|N_G(A\cup B)|\ge r+h$. Suppose $D_2'\cup D_3'\neq \emptyset$. Then we  delete all  vertices in $D_2'\cup D_3'$ from $G$, add the same number of new vertices and connect each new vertex with all the vertices in $A\cup B$. We denote the new graph by $G'$. We will show that $G'$ is $\{H,F\}$-free.
        Suppose there is a copy of $H$ or $F$ in $G'$. Since $|N_G(A\cup B)|\ge r+h$, we can easily find a copy of $H$ or $F$ is $G$, a contradiction. By Claims 1 and 2, $e(G')\ge e(G)$ and then we may assume $|D_2'|=|D_3'|=0$.

         By Claims 1 and 2, $e(G[D_1'])\leq 1$. Suppose $uv\in E(G[D'])$ if $e(G[D_1'])= 1$. For all vertices in $D_1'\setminus (D_0\cup \{u,v\})$, we can add edges between them and $A\cup B$. Then the obtained graph is  $\{H,F\}$-free by the same argument. If $d_{A\cup B}(u)=d_{A\cup B}(v)=\ell-1$, then we are done. And if one of $u,v$, say $u$, with $d_{A\cup B}(u)<\ell-1$, we can delete $uv$ (if it exists) and add all edges between $u$ and $A\cup B$. Then the total number of edges will not decrease. Hence, we may assume $D_1'=D_0$.
    \end{proof}

    \begin{lemma}\label{D3}
        Besides the case when $k=2$ and $I_1=\emptyset$, we may assume $D_3=\emptyset$.
    \end{lemma}
    \begin{proof}
   We first consider the case $I_1\neq \emptyset$. Assume $D_3\neq \emptyset$. Then we  delete all vertices in $D_3$ from $G$, add the same number of new vertices and connect them with all  vertices in $A\cup B$. Since $|N_G(A\cup B)|\ge r+h$, it is easy to check the obtained graph is still $\{H,F\}$-free. According to (\ref{eq: even}) and note that  $(\ell-1-\ell_k-\ell_1+\frac{3}{2})\geq 0$, the number of edges of the obtained graph does not decrease.

    Now we consider the case when $I_1= \emptyset$. By assumption, we just need to consider the case  $k \geq 3$. By (\ref{eq: odd}) and note that $\ell-1-\ell_1-\ell_k+\frac{1}{2}\geq 0$,  we can similarly assume $|D_3|=0$.\end{proof}

    \begin{lemma}\label{lem: D=D0}
        Besides the case when $k=2$ and $I_1=\emptyset$, we may assume $D_2=\emptyset$ and $D_1=D_0$.
    \end{lemma}
    \begin{proof} By Lemma \ref{D3},
         $D_3=\emptyset$. By Lemmas \ref{lem: even, no D E} and \ref{lem: no P3 containing vertex in D E}, we have $d_{D_1\cup  D_2 }(y)\leq 1$ for all $y\in D_1\cup D_2$. Assume there exists $v\in D_1\cup D_2$ with $d_{A\cup B}(v)<\ell-1$.  Since $d_{D_1\cup  D_2 }(v)\leq 1$, we can delete the edge incident to $v$ in $G[D_1\cup  D_2]$ (if it exists) and add all the edges between $v$ and $A\cup B$. The obtained graph is still $\{H,F\}$-free and the number of edges does not decrease.
    \end{proof}
    \begin{lemma}\label{lem: number of edge in D}
        Besides the case when $k=2$ and $I_1=\emptyset$, we may assume $e(G[D_1])\leq 1$.
    \end{lemma}
   \begin{proof}
       According to Lemma \ref{lem: D=D0}, $d_{A\cup B}(v)=|A\cap B|=\ell-1$ for any $v\in D_1$. By Lemmas \ref{lem: even, no D E} and \ref{lem: no P3 containing vertex in D E}, $\Delta(G[D_1])\le1$. If there exist two independent edges $u_1v_1,u_2v_2$ in $G[D_1]$, then $|N(u_1)\cap N(u_2)|=\ell-1\geq \ell_k$. Since $H_k\neq P_3$, we can find a copy of $H_k$, say $H'_k$ contained $v_1,v_2$ as leaf vertices, such that $|V(H'_k)\cap (A\cup B)|=\ell_k-1$, a contradiction with Lemma   \ref{lemma: no H_1 contains l1-1 vertices in AUB}.
   \end{proof}
   \vskip.2cm
   Now we are going to prove our main results.
     \vskip.2cm
    \noindent\textit{Proof of Theorem \ref{thm: linear forest and beta 2}}.
    We first consider the case  $I_1\neq \emptyset$. By Lemmas \ref{lem: even, no D E}, \ref{D3} and \ref{lem: D=D0}, we have
    $$e(G)= e(G[A\cup B])+(\ell-1)(n-\ell+1).$$By Lemma \ref{lemma: common neighbours}, $|A\cup B|=\ell-1$.
    Since $d(y)\ge n-c'$ for any  $y\in A\cup B$, we have $|N(A\cup B)|\ge r+h$. Since $G$ is $\{F,H\}$-free,  we have $e(G[A\cup B])\leq ex(\ell-1,\mathcal{G}_1(F))$. Thus
    $$e(G)\leq (\ell-1)(n-\ell+1)+ex(\ell-1,\mathcal{G}_1(F)).$$

   Now we consider the case $I_1=\emptyset$.  By Lemmas \ref{claim: k=2,D=D0} and \ref{lem: D=D0}, we have  $$e(G)= e(G[A\cup B])+ (\ell-1)(n-\ell+1)+e(G[D_1])$$ and $e(G[D_1])\leq 1$ by Lemma \ref{lem: number of edge in D}. If $e(G[D_1])=0$, we can similarly have $e(G[A\cup B])\leq ex(\ell-1,\mathcal{G}_1(F))$. If $e(G[D_1])=1$, then we have $e(G[A\cup B])\leq ex(\ell-1,\mathcal{G}_2(F)).$
    Above all,
    \begin{equation*}
        \begin{aligned}
            e(G)= &e(G[A\cup B])+e(G[D_1])+(\ell-1)(n-\ell+1)\\
            \leq &(\ell-1)(n-\ell+1)+\max\{ex(\ell-1,\mathcal{G}_1(F)),ex(\ell-1,\mathcal{G}_2(F))+1\}.
        \end{aligned}
    \end{equation*}
    \QEDopen

 \noindent\textit{Proof of Theorem \ref{thm: linear forest and beta 1}}
 Similarly, we have
    $$e(G)= e(G[A\cup B])+e(G[D_1])+\left(\ell-1\right)\left(n-\ell+1\right).$$
    Recall that when $\beta_1(F)=1$, $F$ is a subgrapgh of $B_{r-2}$ where $r=v(F)$ and $B_{r-2}$ is the graph constructed by $r-2$ triangles sharing one edge.
    Suppose $e(G[A\cup B])\geq 1$, say $u_1u_2\in E(G[A\cup B])$.
   Since $|N(A\cup B)|\geq h+r$, $|N(u_1)\cap N(u_2)|\ge h+r$. Thus there is a copy of $B_{r-2}$ in $G$ which implies $G$ has a copy of $F$, a contradiction.

    We claim that if $(F,H)$ has property $\mathscr{R}$, then $e(G[D_1])=0$. By Lemma \ref{lem: even, no D E}, we just need to consider that
     $I_1=\emptyset$. Suppose $e(G[D_1])\geq 1$ and $w_1w_2\in E(G[D_1])$.
    If $r\leq \ell+1$, by Lemmas \ref{claim: k=2,D=D0} and \ref{lem: D=D0}, we can similarly find a copy of $F$ in $G$, a contradiction. Since $(F,H)$ has property $\mathscr{R}$, we have
     $N(F,K_3)=1$ and $f(F)\leq \ell$. Let $v\in A\cup B$. By Lemmas \ref{claim: k=2,D=D0} and \ref{lem: D=D0}, $vw_1w_2v$ is a copy of $K_3$.  Since $f(F)\leq \ell$, $d_G(v)\geq \ell+r$ and $d_G(w_i)\geq |A\cup B|+1=\ell$ for $i=1,2$, we can find a copy of $F$ in $G$, a contradiction. Hence, $e(G[D_1])=0$ if $(F,H)$ has property $\mathscr{R}$. Thus, if  $(F,H)$ has property $\mathscr{R}$, we have
    $$e(G)\leq\left(\ell-1\right)\left(n-\ell+1\right).$$

    \noindent If $(F,H)$ does not have property $\mathscr{R}$, we have
    $$e(G)\leq \left(\ell-1\right)\left(n-\ell+1\right)+1.$$
Thus, the end of the proof.
\QEDopen
\section*{Acknowledgement}
This work is  supported by  the National Natural Science Foundation of China~(Grant 12171272 \& 12426603).

\section*{Declaration of competing interest}
The authors declare that they have no known competing financial interests or personal relationships that could have appeared to influence the work reported in this paper.

\section*{Data availability}
No data was used for the research described in the article.

\end{document}